\documentclass[twoside,12pt,a4paper,leqno,english]{amsart}

\usepackage{amsmath,amsfonts,amsthm,enumitem}
\usepackage[colorlinks=true,linkcolor=blue,citecolor=blue,%
filecolor=blue,urlcolor=blue,pageanchor=true,plainpages=false,%
linktocpage]{hyperref}
\usepackage[english]{babel}

\numberwithin{equation}{section}
\newtheorem{thm}{Theorem}[section]
\newtheorem{prop}[thm]{Proposition}

\newtheorem{lem}[thm]{Lemma}
\theoremstyle{definition}
\newtheorem{defn}[thm]{Definition}
\newtheorem{rem}[thm]{Remark}

\newcommand{\N}{\mathbb{N}}
\newcommand{\Z}{\mathbb{Z}}
\newcommand{\R}{\mathbb{R}}

\newcommand{\cl}[1]{\overline{#1}}

\newcommand{\bd}[1]{\partial #1}

\newcommand{\dvg}{\mathrm{div}\,}

\newcommand{\essinf}{\operatornamewithlimits{\mathrm{ess\,inf}}}

\newcommand{\leb}{\mathcal{L}}

\makeatletter
\newcommand{\mylabel}[2]{#2\def\@currentlabel{#2}\label{#1}}
\makeatother

\newcommand{\hai}{$(a_1)$}
\newcommand{\haii}{$(a_2)$}

\newcommand{\hGi}{$(\mathcal{G}_1)$}
\newcommand{\hGii}{$(\mathcal{G}_2)$}

\newcommand{\hui}{$(ua_1)$}
\newcommand{\huii}{$(ua_2)$}
\newcommand{\huiii}{$(u\mathcal{G}_1)$}
\newcommand{\huiv}{$(u\mathcal{G}_2)$}
\newcommand{\hgi}{$(g_1)$}
\newcommand{\hgii}{$(g_2)$}
\newcommand{\ha}{$(a)$}
\newcommand{\hb}{$(b)$}
\newcommand{\hc}{$(c)$}
\newcommand{\hd}{$(d)$}

\newcommand{\au}[1]{\textsc{#1}}
\newcommand{\titleart}[1]{\textrm{#1}}
\newcommand{\jour}[1]{\textit{#1}}
\newcommand{\volart}[1]{\textbf{#1}}
\newcommand{\no}[1]{\textit{no.} {#1}}

\begin{document}


\title[A
jumping problem with right hand side measure]{A
jumping problem for quasilinear elliptic equations
with right hand side measure}

\author{Michele Colturato}
\address{Dipartimento di Matematica ``F.~Casorati''\\
         Universit\`a di Pavia\\
         Via Ferrata 5\\
         27100 Pavia, Italy}
\email{michele.colturato@unipv.it}
\thanks{The authors are members of the 
   \emph{Gruppo Nazionale per l'Analisi Matematica, la Probabilit\`a
				e le loro Applicazioni} (\emph{GNAMPA}) of the 
   \emph{Istituto Nazionale di Alta Matematica} (\emph{INdAM})}
\author{Marco Degiovanni}
\address{Dipartimento di Matematica e Fisica\\
         Universit\`a Cattolica del Sacro Cuore\\
         Via Trie\-ste 17\\
         25121 Bre\-scia, Italy}
\email{marco.degiovanni@unicatt.it}
				




%
\begin{abstract}
We consider a quasilinear elliptic equation with right hand side 
measure, in which the lower order term has a behavior of 
jumping type.
By means of techniques of degree theory,  we prove the existence
of one or two entropy solutions.
\end{abstract}
\dedicatory{Dedicated to the memory of Antonio Ambrosetti}
\maketitle 


\section{Introduction}
Let $\Omega$ be a bounded and open subset of $\R^N$, 
let $1<p<\infty$ and denote by
$\Delta_p u := \dvg(|\nabla u|^{p-2}\nabla u)$ the
$p$-Laplace operator.
We are interested in the solvability of the problem
\begin{equation}
\label{eq:main}
\begin{cases}
- \Delta_p u = g(x,u,\nabla u) + t \mu_0 + \mu_1
&\qquad\text{in $\Omega$}\,,\\
u=0
&\qquad\text{on $\partial\Omega$}\,,
\end{cases}
\end{equation}
when $\mu_0, \mu_1$ are Radon measures on $\Omega$
and $g$ has a behavior of jumping type.
\par
More precisely, according 
to~\cite{boccardo_gallouet_orsina1996}, we denote 
by $\mathcal{M}_b^p(\Omega)$ the set of Radon measures $\mu$
on~$\Omega$ such that the total variation $|\mu|$ is bounded
and absolutely continuous with respect to the $p$-capacity.
According to~\cite{lindqvist1990}, we also denote by
$\lambda_1$ the first eigenvalue of $-\Delta_p$
with homogeneous Dirichlet boundary condition
and we assume that
\[
g:\Omega\times(\R\times\R^N)\rightarrow\R
\]
is a Carath\'eodory function satisfying:
\emph{
\begin{enumerate}[label={\upshape(\Roman*)}, align=parleft, 
leftmargin=*]
\item[\mylabel{g1}{\hgi}]
there exist $\alpha\in L^1(\Omega)$ and $\beta\in\R$
such that
\[
|g(x,s,\xi)| \leq \alpha(x) + \beta |s|^{p-1} + \beta |\xi|^{p-1}
\]
for a.a. $x\in \Omega$ and all $s\in\R$ and $\xi\in\R^N$;
\item[\mylabel{g2}{\hgii}]
there exist $\underline{\lambda}, \overline{\lambda}\in\R$ 
and a negligible subset $E$ of $\Omega$ such that
\begin{gather*}
\underline{\lambda} < \lambda_1 < \overline{\lambda} \,,\\
\lim_n\,\frac{g(x,\varrho_n s_n,\varrho_n \xi_n)}{\varrho_n^{p-1}} 
= \underline{\lambda} \,(s^+)^{p-1} - 
\overline{\lambda} \,(s^-)^{p-1}
\qquad\text{for all $x\in \Omega\setminus E$}\,,
\end{gather*}
whenever $\varrho_n \to+\infty$, $s_n \to s$ and $\xi_n \to \xi$.
\end{enumerate}
}
An example of function $g$ satisfying the above assumptions
is given by
\[
g(x,s,\xi) = \alpha(x) + \underline{\lambda} \,(s^+)^{p-1} - 
\overline{\lambda} \,(s^-)^{p-1} + g_0(s,\xi) \,,
\]
where $\alpha\in L^1(\Omega)$ and 
$g_0:\R\times\R^N\rightarrow\R$ 
is a continuous function such that
\[
\lim_{|(s,\xi)| \to +\infty}\,\,
\frac{g_0(s,\xi)}{|(s,\xi)|^{p-1}} = 0\,.
\]
Let us state our main result.
\begin{thm}
\label{thm:main}
Assume that $\Omega$ is also connected.
Then, for every 
$\mu_0\in\mathcal{M}_b^p(\Omega)\setminus\{0\}$ with $\mu_0\geq 0$
and for every $\mu_1\in \mathcal{M}_b^p(\Omega)$, there exist
$\underline{t} \leq \overline{t}$ in $\R$ such that 
problem~\eqref{eq:main} admits:
\begin{itemize}
\item
no entropy solution whenever 
$t < \underline{t}$;
\item
at least one entropy solution 
whenever $\underline{t} \leq t \leq \overline{t}$;
\item
at least two entropy solutions whenever $t > \overline{t}$.
\end{itemize}
\end{thm}
When $p=2$, $\mu_0$, $\mu_1$ and $\alpha$ are functions
with a suitable summability and $g$ is independent of $\nabla u$, 
this kind of result (in a sharper form)
goes back to the celebrated paper
of \textsc{Ambrosetti} and 
\textsc{Prodi}~\cite{ambrosetti_prodi1972}.
Several extensions and variants have been considered, 
as long as the principal part of the equation is linear.
We refer the reader e.g.
to~\cite{amann_hess1979, ambrosetti1984, ambrosetti_prodi1993, 
marino_micheletti_pistoia1994, marino_saccon1997}.
\par
The case $p\neq 2$ has been much less studied.
Again, when $\mu_0$, $\mu_1$ and $\alpha$ are functions
with a suitable summability and $g$ is independent of $\nabla u$,
the problem has been treated 
in~\cite{arcoya_ruiz2006, koizumi_schmitt2005}, 
while~\cite{groli_squassina2003} is devoted to a case in
which $g$ may depend on $\nabla u$ and a natural growth is allowed.
Let us point out that, already in the case considered
in~\cite{arcoya_ruiz2006, koizumi_schmitt2005}, it seems to be 
open whether or not $\underline{t} = \overline{t}$.
This fact is proved in~\cite{arcoya_ruiz2006} only when $p > 2$.
\par
On the other hand, also elliptic equations with right hand side
measure have been deeply investigated, starting from the
paper of \textsc{Stampacchia}~\cite{stampacchia1965}.
For the case in which the principal part is linear, we refer the 
reader to~\cite{amann_quittner1998, brezis_marcus_ponce2007,
marcus_veron2014}.
Let us also 
mention~\cite{degiovanni_scaglia2011, ferrero_saccon2007},
where the solutions of the problem are found as 
critical points of suitable functionals.
About jumping problems in this setting, we refer
to~\cite{orsina1993} and the mentioned 
paper~\cite{ferrero_saccon2007}.
\par
Quasilinear elliptic equations with right hand side measure have 
been also widely studied, essentially in the coercive case.
Let us refer e.g.
to~\cite{benilan_boccardo_gallouet_gariepy_pierre_vazquez1995,
boccardo_gallouet_orsina1996, dalmaso_murat_orsina_prignet1999,
greco_iwaniec_sbordone1997, kilpelainen_xu1996}.
The noncoercive case has been faced
in~\cite{colturato_degiovanni2016}, where a degree theory has
been developed to prove a result on the line of the Fredholm 
alternative.
The same degree theory (up to minor changes) will be here
applied to treat the case with jumping nonlinearity.
With respect to the case in which $\mu_0$, $\mu_1$ and $\alpha$ 
are functions with a suitable summability, here we cannot appeal 
to regularity properties, as 
in~\cite{arcoya_ruiz2006, koizumi_schmitt2005}, and the usual
approach, based on operators between spaces in duality, cannot 
be applied.
\par
In the next section, we recall the definition and the main
properties of entropy solutions, as defined
in~\cite{benilan_boccardo_gallouet_gariepy_pierre_vazquez1995,
boccardo_gallouet_orsina1996}, while in 
Section~\ref{sect:degree} the recall, up to a minor variant,
the degree theory developed in~\cite{colturato_degiovanni2016}.
Section~\ref{sect:homogeneous} is devoted to the study of the 
particular case
\[
\begin{cases}
- \Delta_p u = \underline{\lambda} \,(u^+)^{p-1} - 
\overline{\lambda} \,(u^-)^{p-1} +  \mu_0 
&\qquad\text{in $\Omega$}\,,\\
u=0
&\qquad\text{on $\partial\Omega$}\,,
\end{cases}
\]
which allows to treat, in Section~\ref{sect:asympt}, the
problem~\eqref{eq:main} when 
$t<\underline{t}$ or $t>\overline{t}$.
Finally, in Section~\ref{sect:proof} we adapt to our setting an 
idea of~\cite{deuel_hess1976} to treat the case
$\underline{t}\leq t \leq \overline{t}$ and complete the proof 
of Theorem~\ref{thm:main}.


\section{Entropy solutions}
\label{sect:entropy}
From now on, $\Omega$ will denote a bounded and open subset 
of~$\R^N$, while $\leb^N$ will denote the Lebesgue measure in 
$\R^N$ and $s^\pm:=\max\{\pm s,0\}$ the positive and negative 
part of a real number $s$.
Moreover, we denote by $\|~\|_p$ the usual $L^p$-norm.
\par
According 
to~\cite{benilan_boccardo_gallouet_gariepy_pierre_vazquez1995,
boccardo_gallouet_orsina1996, dalmaso_murat_orsina_prignet1999}, 
we also denote by $\mathcal{T}^{1,p}_0(\Omega)$ the set of 
(classes of equivalence of) functions 
$u:\Omega\rightarrow[-\infty,+\infty]$ such that 
$T_k(u)\in W^{1,p}_0(\Omega)$ for all $k>0$, where
\[
T_k(s)=
\begin{cases}
- k 
&\qquad\text{if $s < - k$}\,,\\
s 
&\qquad\text{if $- k \leq s \leq k$}\,,\\
k
&\qquad\text{if $s > k$}\,,
\end{cases}
\]
and such that $\{|u|=+\infty\}$ is negligible.
For every $u\in \mathcal{T}^{1,p}_0(\Omega)$, there exists one 
and only one measurable (class of equivalence)
$\nabla u:\Omega\rightarrow\R^N$ such that
$g(u)\in W^{1,p}_0(\Omega)$ and 
$\nabla[g(u)]=g'(u)\nabla u$ a.e. in $\Omega$,
whenever $g:\R\rightarrow\R$ is Lipschitz continuous with 
$g(0)=0$ and $g'(s)=0$ outside some compact subset of~$\R$.
Moreover, according to~\cite{dalmaso_murat_orsina_prignet1999},
each $u\in \mathcal{T}^{1,p}_0(\Omega)$ admits
a Borel and $\mathrm{cap}_p$-quasi continuous representative
$\tilde{u}:\Omega\rightarrow[-\infty,+\infty]$, defined up to 
a set of null $p$-capacity, which we still denote by $u$. 
Thus, the set $\{|u|=+\infty\}$ has null measure, 
but could have strictly positive $p$-capacity.
\par
Now let $a:\Omega\times\R^N\rightarrow\R^N$ be a Carath\'eodory 
function such that:
\emph{
\begin{enumerate}[label={\upshape(\Roman*)}, align=parleft, 
leftmargin=*]
\item[\mylabel{a1}{\hai}]
there exist $1<p<\infty$, $\alpha_0\in L^1(\Omega)$, 
$\alpha_1\in L^{p'}(\Omega)$, $\beta_1\in\R$ and $\nu>0$
such that
\begin{alignat*}{3}
&a(x,\xi)\cdot\xi \geq \nu |\xi|^p - \alpha_0(x)\,,\\
&|a(x,\xi)| \leq \alpha_1(x) + \beta_1|\xi|^{p-1} \,,
\end{alignat*}
for a.a. $x\in\Omega$ and all $\xi\in\R^N$;
such a $p$ is clearly unique;
\item[\mylabel{a2}{\haii}]
we have
\[
[a(x,\xi)-a(x,\eta)]\cdot(\xi-\eta) >0
\]
for a.a. $x\in\Omega$ and all $\xi,\eta\in\R^N$
with $\xi\neq\eta$.
\end{enumerate}
}
\begin{defn}
Given $\mu\in \mathcal{M}_b^p(\Omega)$, we say that $u$ is an
\emph{entropy solution} of
\begin{equation}
\label{eq:mu}
\begin{cases}
- \dvg[a(x,\nabla u)] = \mu
&\qquad\text{in $\Omega$}\,,\\
u=0
&\qquad\text{on $\partial\Omega$}\,,
\end{cases}
\end{equation}
if $u\in \mathcal{T}^{1,p}_0(\Omega)$ and
\begin{multline*}
\int_\Omega a(x,\nabla u)\cdot\nabla [T_k(u - v)]\,d\leb^N
\leq \int_\Omega T_k(u - v)\,d\mu \qquad
\text{for all $k>0$ and all $v\in C^{\infty}_c(\Omega)$}\,.
\end{multline*}
We also say that $u$ satisfies~\eqref{eq:mu} in the
\emph{entropy sense}.
\end{defn}
\begin{thm}
\label{thm:g=0}
For every $\mu\in \mathcal{M}_b^p(\Omega)$, there exists one 
and only one entropy solution~$u$ of~\eqref{eq:mu} and it turns 
out that $\{|u|=+\infty\}$ has null $p$-capacity.
\par
Moreover, if $\mu_1, \mu_2\in \mathcal{M}_b^p(\Omega)$ and 
$u_1, u_2\in \mathcal{T}^{1,p}_0(\Omega)$ are the corresponding
entropy solutions of~\eqref{eq:mu}, then we have
\begin{multline*}
\int_{\{0 \leq u_1 - u_2 \leq k\}}
[a(x,\nabla u_1) - a(x,\nabla u_2)]\cdot
(\nabla u_1 - \nabla u_2)\,d\leb^N \\
\leq
\int_\Omega [T_k(u_1 - u_2)]^+\,d\mu_1
- \int_\Omega [T_k(u_1 - u_2)]^+\,d\mu_2
\qquad\text{for all $k>0$}\,.
\end{multline*}
Finally, if
\[
\int_\Omega v\,d\mu_1 \leq \int_\Omega v\,d\mu_2
\qquad\text{for all $v\in C^{\infty}_c(\Omega)$ with $v\geq 0$}\,,
\]
then we have $u_1\leq u_2$ a.e. in $\Omega$.
\end{thm}
\begin{proof}
Up to a minor detail, concerning the presence of $\alpha_0$ in
assumption~\ref{a1},
the existence and uniqueness has been proved 
in~\cite{benilan_boccardo_gallouet_gariepy_pierre_vazquez1995,
boccardo_gallouet_orsina1996} for all $p>1$.
The case $p > 2 - 1/N$ has been treated also
in~\cite{kilpelainen_xu1996}, where the inequality is proved.
\par
We also refer to the proof 
of~\cite[Theorem~2.8]{colturato_degiovanni2016}.
\end{proof}
Let us also recall a result in the spirit 
of~\cite{brezis_browder1978}.
\begin{thm}
\label{thm:bb}
Let $\mu\in \mathcal{M}_b^p(\Omega)$ and let $u$ be the 
entropy solution of~\eqref{eq:mu}.
Then we have
\begin{multline*}
\int_\Omega (a(x,\nabla u)\cdot\nabla v)^+\,d\leb^N
+ \int_\Omega (v \nu)^-\,d|\mu| \\
= \int_\Omega (a(x,\nabla u)\cdot\nabla v)^-\,d\leb^N
+ \int_\Omega (v \nu)^+\,d|\mu|
\qquad\text{for all $v\in \mathcal{T}^{1,p}_0(\Omega)$}\,,
\end{multline*}
where $d\mu = \nu d|\mu|$ and $\nu$ is a Borel function
with $|\nu|=1$ $|\mu|$-a.e. in $\Omega$
(both sides of the equality could be $+\infty$).
\par
In particular, it holds
\begin{multline*}
\int_\Omega (a(x,\nabla u)\cdot\nabla v)^+\,d\leb^N 
= \int_\Omega (a(x,\nabla u)\cdot\nabla v)^-\,d\leb^N
+ \int_\Omega v \,d\mu \\
\quad\text{for all 
$v\in W^{1,p}_0(\Omega) \cap L^\infty(\Omega)$}\,.
\end{multline*}
\end{thm}
\begin{proof}
See~\cite[Theorem~2.5]{colturato_degiovanni2016}.
\end{proof}
Since we aim to add lower order terms in the 
problem~\eqref{eq:mu}, it is convenient to introduce
an intermediate space between $W^{1,p}_0(\Omega)$ and
$\mathcal{T}^{1,p}_0(\Omega)$, as suggested by the 
results of~\cite{boccardo_gallouet1996}.
\par
Let us denote by $\varphi_p:\R\rightarrow\R$ the 
increasing $C^\infty$-diffeomorphism such that
\[
\varphi_p'(s) = \frac{1}{\{(1+s^2)[\log(e+s^2)]^4
\}^{\frac{1}{2p}}}
\,,\qquad\varphi_p(0)=0\,.
\]
It is easily seen that the map  
$\left\{u\mapsto \varphi_p(u)\right\}$ is bijective from 
$\mathcal{T}^{1,p}_0(\Omega)$ onto 
$\mathcal{T}^{1,p}_0(\Omega)$  
and that $\nabla[\varphi_p(u)] = \varphi_p'(u)\nabla u$
a.e. in $\Omega$.
According to~\cite{colturato_degiovanni2016}, we set
\[
\Phi^{1,p}_0(\Omega) = \left\{u\in \mathcal{T}^{1,p}_0(\Omega):
\,\,\varphi_p(u)\in W^{1,p}_0(\Omega)\right\}\,.
\]
It turns out that
\[
W^{1,p}_0(\Omega) \subseteq \Phi^{1,p}_0(\Omega)
\]
and that the set $\{|u|=+\infty\}$ has null $p$-capacity,
whenever $u\in\Phi^{1,p}_0(\Omega)$.
Moreover, we have $\Phi^{1,p}_0(\Omega)=W^{1,p}_0(\Omega)$
when $p>N$.
\par
Since $\left\{u\mapsto \varphi_p(u)\right\}$
is bijective also from $\Phi^{1,p}_0(\Omega)$
onto $W^{1,p}_0(\Omega)$, there is a natural structure of
complete metric space on $\Phi^{1,p}_0(\Omega)$ which makes 
$\left\{u\mapsto \varphi_p(u)\right\}$ an isometry.
In particular, the distance function on $\Phi^{1,p}_0(\Omega)$ 
is given by
\[
d(u,v) = \|\nabla[\varphi_p(u)]-\nabla[\varphi_p(v)]\|_p
\qquad\text{for all $u,v \in \Phi^{1,p}_0(\Omega)$}\,.
\]
\begin{prop}
\label{prop:tPhi}
For every $u \in\mathcal{T}^{1,p}_0(\Omega)$, we have
\[
u \in\Phi^{1,p}_0(\Omega) \quad\Longleftrightarrow\quad
\varphi_p'(u)|\nabla u| \in L^p(\Omega)\,.
\]
Moreover, we have
\begin{gather*}
\text{$\nabla u = \nabla v$ a.e. in $\Omega$}\quad
\Longrightarrow\quad
\text{$u = v$ a.e. in $\Omega$}
\qquad\text{for all $u,v \in\Phi^{1,p}_0(\Omega)$}\,,\\
\max\{u,v\}\,,\,\,\min\{u,v\}\,,\,\,
tu\in\Phi^{1,p}_0(\Omega)
\qquad\text{for all $u,v \in\Phi^{1,p}_0(\Omega)$ and $t\in\R$}\,,
\end{gather*}
and the map $\left\{t \mapsto t u\right\}$
is continuous from $\R$ into $\Phi^{1,p}_0(\Omega)$.
Finally, it holds
\begin{multline*}
\|\nabla[\varphi_p(t u)]\|_p \leq
|t|\,\left\{\frac{\left[1+2(\log |t|)^-\right]^2}{
\min\{|t|,1\}}\right\}^{1/p}\,
\|\nabla[\varphi_p(u)]\|_p \\
\qquad\text{for all $u\in \Phi^{1,p}_0(\Omega)$ and $t\neq 0$}\,.
\end{multline*}
\end{prop}
\begin{proof}
We have already recalled that 
$\nabla[\varphi_p(u)] = \varphi_p'(u)\nabla u$, whence
$\varphi_p'(u)|\nabla u| \in L^p(\Omega)$
whenever $u \in\Phi^{1,p}_0(\Omega) $.
Conversely, for every $k>0$ we have 
$\varphi_p(T_k(u))\in W^{1,p}_0(\Omega)$ and
\[
\varphi_p'(T_k(u)) |\nabla T_k(u)| \leq
\varphi_p'(u) |\nabla u| \,.
\]
If $\varphi_p'(u)|\nabla u| \in L^p(\Omega)$, the sequence 
$(\varphi_p(T_k(u)))$ is bounded in $W^{1,p}_0(\Omega)$, 
whence $\varphi_p(u)\in W^{1,p}_0(\Omega)$.
\par
Since $\varphi_p$ is strictly increasing, 
it is easily seen that 
$\max\{u,v\},\min\{u,v\}\in\Phi^{1,p}_0(\Omega)$ and it follows 
from~\cite[Proposition~2.3]{colturato_degiovanni2016} that
\[
\text{$\nabla u = \nabla v$ a.e. in $\Omega$}\quad
\Longrightarrow\quad
\text{$u = v$ a.e. in $\Omega$}
\qquad\text{for all $u,v \in\Phi^{1,p}_0(\Omega)$}
\]
(while this is not true if one only knows that
$u,v \in \mathcal{T}^{1,p}_0(\Omega)$).
\par
Moreover, it holds $t u \in\mathcal{T}^{1,p}_0(\Omega)$ 
for all $t\in\R$.
If $|t|\geq 1$, we also have
\[
\varphi_p'(t u) |\nabla(t u)| =
|t|\,\varphi_p'(|t| u) |\nabla u| \leq
|t|\,\varphi_p'(u) |\nabla u| \,,
\]
as $\varphi_p'$ is increasing on $]-\infty,0]$
and decreasing on $[0,+\infty[$.
It follows that $t u\in\Phi^{1,p}_0(\Omega)$, that
\[
\|\nabla[\varphi_p(t u)]\|_p \leq
|t|\,\|\nabla[\varphi_p(u)]\|_p
\]
and that the map $\left\{t \mapsto t u\right\}$
is continuous from $\R\setminus]-1,1[$ into 
$\Phi^{1,p}_0(\Omega)$.
\par
On the other hand, $0 \in \Phi^{1,p}_0(\Omega)$ and,
if $0<|t|\leq 1$, we have
\[
\frac{1+s^2}{1+t^2s^2} \leq \frac{1}{t^2}\,,\qquad
\frac{\log(e+s^2)}{\log(e+t^2s^2)} \leq 1 +
\frac{- \log (t^2)}{\log(e+t^2s^2)} \leq
1 - \log (t^2)\,,
\]
whence
\[
[\varphi_p'(ts)]^p \leq 
\frac{(1-2\log |t|)^2}{|t|}\,[\varphi_p'(s)]^p\,,
\]
which implies
\[
\int_\Omega [\varphi_p'(tu)]^p|\nabla(tu)|^p\,d\leb^N \leq
|t|^{p-1}\,(1-2\log |t|)^2\,
\int_\Omega [\varphi_p'(u)]^p|\nabla(u)|^p\,d\leb^N \,.
\]
Therefore $t u\in\Phi^{1,p}_0(\Omega)$,
\[
\|\nabla[\varphi_p(t u)]\|_p \leq
|t|\,\left\{\frac{\left(1-2\log |t|\right)^2}{|t|}\right\}^{1/p}\,
\|\nabla[\varphi_p(u)]\|_p
\]
and the map $\left\{t \mapsto t u\right\}$
is continuous from $[-1,1]$ into $\Phi^{1,p}_0(\Omega)$.
\end{proof}
\begin{prop}
\label{prop:Phi}
The following facts hold:
\begin{enumerate}[label={\upshape\alph*)}, align=parleft, 
widest=(a), leftmargin=*]
\item[\mylabel{Phia}{\ha}]
if $u\in \Phi^{1,p}_0(\Omega)$, then 
$|\nabla u|^{p-1}\in L^1(\Omega)$ and
$|u|^{p-1} \in L^1(\Omega)$;
\item[\mylabel{Phib}{\hb}]
if $(u_n)$ is bounded in $\Phi^{1,p}_0(\Omega)$,
then $(|\nabla u_n|^{p-1})$ is bounded in $L^1(\Omega)$;
moreover, there exists $u\in\Phi^{1,p}_0(\Omega)$ such that, 
up to a subsequence, $(|u_n|^{p-2}\,u_n)$ is strongly convergent 
to $|u|^{p-2}\,u$  in $L^1(\Omega)$;
\item[\mylabel{Phic}{\hc}]
if $u_n, u \in \Phi^{1,p}_0(\Omega)$,
$(u_n)$ is bounded in $\Phi^{1,p}_0(\Omega)$ and
$(\nabla u_n)$ is convergent to $\nabla u$ in measure, 
then we have: 
\[
\null\qquad\qquad
\lim_n |\nabla u_n|^{p-2}\,\nabla u_n 
= |\nabla u|^{p-2}\,\nabla u \qquad
\text{strongly in $L^1(\Omega;\R^N)$}\,;
\]
\item[\mylabel{Phid}{\hd}]
if $(u_n)$ is convergent to $u$ in 
$\Phi^{1,p}_0(\Omega)$, then we have
\[
\lim_n T_k(u_n) = T_k(u) 
\qquad\text{strongly in $W^{1,p}_0(\Omega)$, for all $k>0$}\,.
\]
\end{enumerate}
\end{prop}
\begin{proof}
It is a particular case 
of~\cite[Proposition~2.9]{colturato_degiovanni2016}.
\end{proof}
\begin{thm}
\label{thm:regentr}
If $\mu\in \mathcal{M}_b^p(\Omega)$ and $u$ is the entropy 
solution of~\eqref{eq:mu}, then $u\in \Phi^{1,p}_0(\Omega)$.
Moreover, if we define an increasing and bounded 
$C^\infty$-function $\psi:\R\rightarrow\R$ by
\[
\psi'(s) = 
\frac{1}{\{(1+s^2)[\log(e+s^2)]^4
\}^{\frac{1}{2}}}
= [\varphi_p'(s)]^p
\,,\qquad\psi(0)=0\,,
\]
then 
$\psi(u)\in W^{1,p}_0(\Omega)\cap L^{\infty}(\Omega)$,
$\psi'(u)\,a(x,\nabla u)\cdot \nabla u\in L^1(\Omega)$
and
\begin{gather}
\label{eq:sol}
\int_{\Omega} 
[\varphi_p'(u)]^p\,a(x,\nabla u)\cdot \nabla u\,d\leb^N =
\int_{\Omega} \psi'(u)\,a(x,\nabla u)\cdot \nabla u\,d\leb^N =
\int_\Omega \psi(u)\,d\mu\,,\\
\label{eq:estimate}
\nu \int_\Omega |\nabla [\varphi_p(u)]|^p\,d\leb^N \leq
\|\psi\|_\infty\,|\mu|(\Omega) + \|\alpha_0\|_1\,.
\end{gather}
\end{thm}
\begin{proof}
See~\cite[Theorem~2.6]{colturato_degiovanni2016}.
\end{proof}
\begin{rem}
\label{rem:equiv}
By Theorem~\ref{thm:regentr}, in the definition of
entropy solution it is equivalent to require
$u\in \mathcal{T}^{1,p}_0(\Omega)$ or
$u\in \Phi^{1,p}_0(\Omega)$.
\par
Let us also recall that, according
to~\cite[Theorem~2.1]{boccardo_gallouet_orsina1996}, for every
$\mu\in\mathcal{M}_b^p(\Omega)$ there exist
$w^{(0)}\in L^1(\Omega)$ and $w^{(1)}\in L^{p'}(\Omega;\R^N)$
such that $\mu = w^{(0)} - \dvg w^{(1)}$, namely
\[
\int_\Omega v\,d\mu =
\int_\Omega v\,w^{(0)}\,d\leb^N
+ \int_\Omega (\nabla v)\cdot w^{(1)}\,d\leb^N \qquad
\text{for all $v\in C^\infty_c(\Omega)$}\,.
\]
Moreover, $w^{(0)}$ and $w^{(1)}$ can be chosen so that
$\|w^{(0)}\|_1 \leq |\mu|(\Omega)$ and $\|w^{(1)}\|_{p'} \leq 1$.
Actually, this last fact is not stated, by it is clear from 
the proof of~\cite[Theorem~2.1]{boccardo_gallouet_orsina1996}.
\end{rem}
\begin{thm}
\label{thm:compres}
Let $\mu_n, \mu\in \mathcal{M}_b^p(\Omega)$ and let $u_n, u$ 
be the corresponding entropy solutions of~\eqref{eq:mu}.
Assume that $\mu_n = w_n^{(0)} - \dvg w_n^{(1)}$ and
$\mu = w^{(0)} - \dvg w^{(1)}$, 
where $(w_n^{(0)})$ is weakly convergent to $w^{(0)}$
in $L^1(\Omega)$ and $(w_n^{(1)})$ is strongly convergent 
to~$w^{(1)}$ in $L^{p'}(\Omega;\R^N)$.
\par
Then $(u_n)$ is convergent to $u$ in $\Phi^{1,p}_0(\Omega)$.
\end{thm}
\begin{proof}
See~\cite[Lemma~5.4]{colturato_degiovanni2016}.
\end{proof}
%


\section{A degree for a class of quasilinear elliptic equations}
\label{sect:degree}
Consider again a Carath\'eodory function
$a:\Omega\times\R^N\rightarrow\R^N$ 
satisfying~\ref{a1} and~\ref{a2}.
Consider also a map
\[
\mathcal{G}:\Phi^{1,p}_0(\Omega) \longrightarrow L^1(\Omega)
\]
such that:
\emph{
\begin{enumerate}[label={\upshape(\Roman*)}, align=parleft, 
leftmargin=*]
\item[\mylabel{G1}{\hGi}]
the map $\mathcal{G}$ is bounded on bounded subsets;
\item[\mylabel{G2}{\hGii}]
if $u_n, u \in \Phi^{1,p}_0(\Omega)$,
$(u_n)$ is bounded in $\Phi^{1,p}_0(\Omega)$ and
$(u_n,\nabla u_n)$ is convergent to $(u,\nabla u)$ in measure, 
then $(\mathcal{G}(u_n))$ is weakly convergent
to $\mathcal{G}(u)$ in $L^1(\Omega)$.
\end{enumerate}
}
Taking into account Proposition~\ref{prop:Phi}, a first
example is given by
\[
\mathcal{G}(u) = g(x,u,\nabla u)\,,
\]
where 
\[
g:\Omega\times(\R\times\R^N)\rightarrow\R
\]
is a Carath\'eodory function satisfying~\ref{g1}.
On the other hand, if $\mathcal{G}$
satisfies~\ref{G1} and ~\ref{G2} and
$\underline{u}\in \Phi^{1,p}_0(\Omega)$, then
\[
\widetilde{\mathcal{G}}(u) =
\mathcal{G}(\max\{u,\underline{u}\})
\]
also satisfies~\ref{G1} and ~\ref{G2}.
\par
The next notion is suggested by Remark~\ref{rem:equiv}.
\begin{defn}
\label{defn:entrgen}
If $\mu\in\mathcal{M}_b^p(\Omega)$, we say that
$u$ is an \emph{entropy solution} of
\begin{equation}
\label{eq:bmu}
\begin{cases}
- \dvg[a(x,\nabla u)] - \mathcal{G}(u) = \mu 
&\qquad\text{in $\Omega$}\,,\\
u=0
&\qquad\text{on $\partial\Omega$}\,,
\end{cases}
\end{equation}
if $u\in \Phi^{1,p}_0(\Omega)$ and
\begin{multline*}
\int_\Omega a(x,\nabla u)\cdot\nabla [T_k(u - v)]\,d\leb^N
- \int_\Omega \mathcal{G}(u)\,T_k(u - v)\,d\leb^N \\
\leq \int_\Omega T_k(u-v)\,d\mu
\qquad\text{for all $k>0$ and all $v\in C^{\infty}_c(\Omega)$}\,.
\end{multline*}
Again, we also say that $u$ satisfies~\eqref{eq:bmu} in the
\emph{entropy sense}.
\end{defn}
We will also consider parametric problems in which
\[
a:\Omega\times(\R^N\times [0,1])\rightarrow \R^N\,,\qquad
\mathcal{G}:\Phi^{,p}_0(\Omega)\times[0,1]\rightarrow L^1(\Omega)
\]
satisfy~\ref{a1}, \ref{a2}, \ref{G1} and~\ref{G2}
uniformly, namely:
\emph{
\begin{enumerate}[label={\upshape(\Roman*)}, align=parleft, 
leftmargin=*]
\item[\mylabel{u1}{\hui}]
the function $a$ is Carath\'eodory and
there exist $1<p<\infty$, $\alpha_0\in L^1(\Omega)$, 
$\alpha_1\in L^{p'}(\Omega)$, $\beta_1\in\R$ and $\nu>0$
such that
\begin{alignat*}{3}
&a_\tau(x,\xi)\cdot\xi \geq \nu |\xi|^p - \alpha_0(x)\,,\\
&|a_\tau(x,\xi)| \leq \alpha_1(x) + \beta_1|\xi|^{p-1} \,,
\end{alignat*}
for a.a. $x\in\Omega$ and all $\xi\in\R^N$ and $\tau\in [0,1]$;
\item[\mylabel{u2}{\huii}]
we have
\[
[a_\tau(x,\xi)-a_\tau(x,\eta)]\cdot(\xi-\eta) >0
\]
for a.a. $x\in\Omega$ and all $\xi,\eta\in\R^N$ and
$\tau\in [0,1]$ with $\xi\neq\eta$;
\item[\mylabel{u3}{\huiii}]
the map $\mathcal{G}$ is bounded on $B\times [0,1]$,
whenever $B$ is bounded in $\Phi^{1,p}_0(\Omega)$;
\item[\mylabel{u4}{\huiv}]
if $(\tau_n)$ is convergent to $\tau$ in $[0,1]$,
if $u_n, u \in \Phi^{1,p}_0(\Omega)$,
$(u_n)$ is bounded in $\Phi^{1,p}_0(\Omega)$ and
$(u_n,\nabla u_n)$ is convergent to $(u,\nabla u)$ in measure, 
then $(\mathcal{G}_{\tau_n}(u_n))$ is weakly convergent
to $\mathcal{G}_\tau(u)$ in $L^1(\Omega)$
(we write $a_\tau(x,\xi)$, $\mathcal{G}_\tau(u)$ instead of 
$a(x,(\xi,\tau))$, $\mathcal{G}(u,\tau)$).
\end{enumerate}
}
Let us first recall some standard notions (see
e.g.~\cite{browder1983, oregan_cho_chen2006, skrypnik1994}).
\begin{defn}
\label{defn:S+}
Let $X$ be a reflexive real Banach space and let
$D\subseteq X$.
A map $F:D\rightarrow X'$ is said to be \emph{of class~$(S)_+$} 
if, for every sequence $(u_n)$ in $D$ weakly converging to some 
$u$ in $X$ with
\[
\limsup_n \, \langle F(u_n),u_n-u\rangle \leq 0\,,
\]
it holds $\|u_n-u\|\to 0$.
\end{defn}
\begin{defn}
We denote by $L^1(\Omega)\cap W^{-1,p'}(\Omega)$ the set of
$u$'s in $L^1(\Omega)$ such that
\[
\sup \left\{\,\left|\int_\Omega v u\,d\leb^N\right|:\,\,
v\in C^\infty_c(\Omega)\,,\,\,\|\nabla v\|_p \leq 1\right\}
< +\infty
\]
and we denote by $\mathcal{M}_b^p(\Omega)\cap W^{-1,p'}(\Omega)$ 
the set of $\mu$'s in $\mathcal{M}_b^p(\Omega)$ such that
\[
\sup \left\{\,\left|\int_\Omega v \,d\mu\right|:\,\,
v\in C^\infty_c(\Omega)\,,\,\,\|\nabla v\|_p \leq 1\right\}
< +\infty \,.
\]
\end{defn}
Because of the density of $C^\infty_c(\Omega)$ in
$W^{1,p}_0(\Omega)$, each $u\in L^1(\Omega)\cap W^{-1,p'}(\Omega)$
identifies a unique element in $W^{-1,p'}(\Omega)$, which is
still denoted by $u$, so that
\[
\langle u, v\rangle = \int_\Omega v u\,d\leb^N
\qquad\text{for all $v\in C^\infty_c(\Omega)$}\,,
\]
and each $\mu\in \mathcal{M}_b^p(\Omega\cap W^{-1,p'}(\Omega)$
identifies a unique element in $W^{-1,p'}(\Omega)$, which is
still denoted by $\mu$, so that
\[
\langle \mu, v\rangle = \int_\Omega v \,d\mu
\qquad\text{for all $v\in C^\infty_c(\Omega)$}\,.
\]
However, it may happen that 
$u \in L^1(\Omega)\cap W^{-1,p'}(\Omega)$ and 
$u \not\in L^q(\Omega)$ for all $q>1$.
\par
According
to~\cite{browder1983, oregan_cho_chen2006, skrypnik1994},
for continuous maps of class~$(S)_+$ it is possible to
define a topological degree, which will be denoted by
$\mathrm{deg}_{(S)_+}$.
Starting from this fact, it is possible to define a topological 
degree
\[
\mathrm{deg}(- \dvg[a(x,\nabla u)] 
- \mathcal{G}(u),U,\mu)\in \Z
\]
whenever $\mu\in\mathcal{M}_b^p(\Omega)$ and $U$
is a bounded and open subset of $\Phi^{1,p}_0(\Omega)$
such that~\eqref{eq:bmu} has no entropy solution
$u\in\partial U$.
This has been proved 
in~\cite{colturato_degiovanni2016}, when
\[
\mathcal{G}(u) = g(x,u,\nabla u)
\]
and $g$ is a Carath\'eodory function subjected
to a suitable growth condition (by the way, more
general than~\ref{g1}).
However the extension to our case is straightforward.
\par
Let us recall the main properties of the degree.
\begin{thm}
\label{thm:consistency}
\textbf{\emph{(Consistency property)}}
Suppose that 
$\mu\in\mathcal{M}_b^p(\Omega)\cap W^{-1,p'}(\Omega)$ and that
\[
\mathcal{G}(u) = g(x,u,\nabla u)\,,
\]
where $g$ is a Carath\'eodory function satisfying~\ref{g1}
with $\alpha\in L^1(\Omega)\cap W^{-1,p'}(\Omega)$.
\par
Then the following facts hold:
\begin{enumerate}[label={\upshape\alph*)}, align=parleft, 
widest=(a), leftmargin=*]
\item[\mylabel{consistencya}{\ha}]
for every $u,v\in W^{1,p}_0(\Omega)$ we have
\[
g(x,u,\nabla u)v\in L^1(\Omega)\,,
\qquad
g(x,u,\nabla u)\in L^1(\Omega)\cap W^{-1,p'}(\Omega)\,, 
\]
and the map
\[
\begin{array}{ccc}
W^{1,p}_0(\Omega) & \longrightarrow & W^{-1,p'}(\Omega) \\
\noalign{\medskip}
u & \mapsto & -\dvg[a(x,\nabla u)]-g(x,u,\nabla u)
\end{array}
\]
is continuous and of class~$(S)_+$;
\item[\mylabel{consistencyb}{\hb}]
every entropy solution of~\eqref{eq:bmu} belongs to 
$W^{1,p}_0(\Omega)$ and every $u\in W^{1,p}_0(\Omega)$
is an entropy solution of~\eqref{eq:bmu} if and only if
\[
- \dvg[a(x,\nabla u)] - g(x,u,\nabla u)=\mu
\qquad\text{in $W^{-1,p'}(\Omega)$}\,;
\]
\item[\mylabel{consistencyc}{\hc}]
if $U$ is a bounded and open subset of $\Phi^{1,p}_0(\Omega)$
such that~\eqref{eq:bmu} has no entropy solution  
$u\in\partial U$, then the set
\[
\left\{u\in U:\,\,
- \dvg[a(x,\nabla u)] - g(x,u,\nabla u)=\mu\right\}
\]
is strongly compact in $W^{1,p}_0(\Omega)$ and we have
\begin{multline*}
\null\qquad\qquad
\mathrm{deg}(- \dvg[a(x,\nabla u)] 
- g(x,u,\nabla u),U,\mu) \\
= \mathrm{deg}_{(S)_+}(- \dvg[a(x,\nabla u)] 
- g(x,u,\nabla u),U\cap V,\mu)\,,
\end{multline*}
whenever $V$ is a bounded and open subset of $W^{1,p}_0(\Omega)$
such that there are no entropy solutions of~\eqref{eq:bmu}
in $U\setminus V$.
\end{enumerate}
\end{thm}
\begin{thm}
\label{thm:normalization}
\textbf{\emph{(Normalization property)}}
Let $\mu\in\mathcal{M}_b^p(\Omega)$ and let $U$ be any bounded 
and open subset of $\Phi^{1,p}_0(\Omega)$ containing the 
entropy solution $u$ of 
\[
\begin{cases}
- \dvg[a(x,\nabla u)] = \mu 
&\qquad\text{in $\Omega$}\,,\\
u=0
&\qquad\text{on $\partial\Omega$}\,.
\end{cases}
\]
Then
\[
\mathrm{deg}(- \dvg[a(x,\nabla u)],U,\mu) = 1\,.
\]
\end{thm}
\begin{thm}
\label{thm:existence}
\textbf{\emph{(Solution property)}}
Let $\mu\in\mathcal{M}_b^p(\Omega)$ and let $U$ be a bounded 
and open subset of $\Phi^{1,p}_0(\Omega)$ such 
that~\eqref{eq:bmu}
has no entropy solution $u\in\overline{U}$.
\par
Then
\[
\mathrm{deg}(- \dvg[a(x,\nabla u)]
- \mathcal{G}(u),U,\mu) = 0\,.
\]
\end{thm}
\begin{thm}
\label{thm:addexc}
\textbf{\emph{(Additivity-Excision property)}}
Let $\mu\in\mathcal{M}_b^p(\Omega)$, let $U$ be a bounded 
and open subset of $\Phi^{1,p}_0(\Omega)$ and let $U_1$, $U_2$
be two disjoint open subsets of~$U$.
Assume that~\eqref{eq:bmu} has no entropy solution 
$u\in \cl{U}\setminus(U_1\cup U_2)$.
\par
Then
\begin{multline*}
\mathrm{deg}(- \dvg[a(x,\nabla u)] 
- \mathcal{G}(u),U,\mu) \\
= \mathrm{deg}(- \dvg[a(x,\nabla u)] 
- \mathcal{G}(u),U_1,\mu) \\
+ \mathrm{deg}(- \dvg[a(x,\nabla u)] 
- \mathcal{G}(u),U_2,\mu)\,.
\end{multline*}
\end{thm}
\begin{lem}
\label{lem:proper}
Assume that
\[
a:\Omega\times(\R^N\times [0,1])\rightarrow \R^N\,,\qquad
\mathcal{G}:\Phi^{1,p}_0(\Omega)\times[0,1]\rightarrow L^1(\Omega)
\]
satisfy~\ref{u1}, \ref{u2}, \ref{u3} and~\ref{u4}.
Let $\tau_n \to \tau$ in $[0,1]$, let
$\mu_n, \mu \in\mathcal{M}_b^p(\Omega)$ 
and let $(u_n)$ be a bounded sequence in $\Phi^{1,p}_0(\Omega)$
such that each $u_n$ satisfies
\[
\begin{cases}
- \dvg[a_{\tau_n}(x,\nabla u_n)] 
- \mathcal{G}_{\tau_n}(u_n) = \mu_n
&\quad\text{in $\Omega$}\,,\\
u_n=0
&\quad\text{on $\partial\Omega$}\,,
\end{cases}
\]
in the entropy sense.
Suppose also that
$\mu_n = w_n^{(0)} - \dvg w_n^{(1)}$ and
$\mu = w^{(0)} - \dvg w^{(1)}$,
where $(w_n^{(0)})$ is weakly convergent to $w^{(0)}$
in $L^1(\Omega)$ and $(w_n^{(1)})$ is strongly convergent 
to~$w^{(1)}$ in $L^{p'}(\Omega;\R^N)$.
\par
Then, up to a subsequence, $(u_n)$ is convergent in 
$\Phi^{1,p}_0(\Omega)$ to some $u$ satisfying
\[
\begin{cases}
- \dvg[a_{\tau}(x,\nabla u)] 
- \mathcal{G}_{\tau}(u) = \mu
&\quad\text{in $\Omega$}\,,\\
u=0
&\quad\text{on $\partial\Omega$}\,,
\end{cases}
\]
in the entropy sense.
\end{lem}
\begin{proof}
Since $(u_n)$ is bounded in $\Phi^{1,p}_0(\Omega)$,
from~\ref{u3} we infer that $(\mathcal{G}_{\tau_n}(u_n))$
is bounded in $L^1(\Omega)$.
By~\cite[Lemma~5.3]{colturato_degiovanni2016}, it follows
that there exists $u\in \Phi^{1,p}_0(\Omega)$ such that
$(u_n,\nabla u_n)$ is convergent, up to a subsequence,
to $(u,\nabla u)$ in measure.
From~\ref{u4} we infer that $(\mathcal{G}_{\tau_n}(u_n))$
is weakly convergent to $\mathcal{G}_\tau(u)$ in $L^1(\Omega)$.
By~\cite[Theorem~5.1]{colturato_degiovanni2016}, it follows
that $(u_n)$ is convergent in $\Phi^{1,p}_0(\Omega)$ to $u$
and that $u$ satifies
\[
\begin{cases}
- \dvg[a_{\tau}(x,\nabla u)] 
- \mathcal{G}_{\tau}(u) = \mu
&\quad\text{in $\Omega$}\,,\\
u=0
&\quad\text{on $\partial\Omega$}\,,
\end{cases}
\] 
in the entropy sense. 
\end{proof}
\begin{thm}
\label{thm:homotopy}
\textbf{\emph{(Homotopy invariance property)}}
Assume that
\[
a:\Omega\times(\R^N\times [0,1])\rightarrow \R^N\,,\qquad
\mathcal{G}:\Phi^{,p}_0(\Omega)\times[0,1]\rightarrow L^1(\Omega)
\]
satisfy~\ref{u1}, \ref{u2}, \ref{u3} and~\ref{u4} and let
$\mu_0, \mu_1\in\mathcal{M}_b^p(\Omega)$.
\par
Then the following facts hold:
\begin{enumerate}[label={\upshape\alph*)}, align=parleft, 
widest=(a), leftmargin=*]
\item[\mylabel{homotopya}{\ha}]
for every bounded and closed subset $C$ of 
$\Phi^{1,p}_0(\Omega)$, the set of $(\tau,t)$
in $[0,1]\times[0,1]$ such that
\begin{equation}
\label{eq:bmut}
\begin{cases}
- \dvg[a_\tau(x,\nabla u)] - \mathcal{G}_\tau(u)=
(1-t)\mu_0 + t\mu_1
&\quad\text{in $\Omega$}\,,\\
u=0
&\quad\text{on $\partial\Omega$}\,,
\end{cases}
\end{equation}
admits an entropy solution $u\in C$ is closed in 
$[0,1]\times[0,1]$;
\item[\mylabel{homotopyb}{\hb}]
for every bounded and open subset $U$ of $\Phi^{1,p}_0(\Omega)$,
if~\eqref{eq:bmut} has no entropy solution 
with $(\tau,t)\in[0,1]\times[0,1]$ and $u\in\partial U$, then
\[
\mathrm{deg}(- \dvg[a_\tau(x,\nabla u)]
- \mathcal{G}_\tau(u),U,(1-t)\mu_0+t\mu_1)
\]
is independent of $(\tau,t)\in[0,1]\times[0,1]$.
\end{enumerate}
\end{thm}
\begin{proof}
Assertion~\ref{homotopya} easily follows 
from Lemma~\ref{lem:proper}.
Assume now that~$U$ is a bounded and open subset of 
$\Phi^{1,p}_0(\Omega)$ and that~\eqref{eq:bmut} has no entropy 
solution with $(\tau,t)\in[0,1]\times[0,1]$ 
and $u\in\partial U$.
A straightforward extension 
of~\cite[Theorem~3.7]{colturato_degiovanni2016} shows that
\[
\mathrm{deg}(- \dvg[a_\tau(x,\nabla u)]
- \mathcal{G}_\tau(u),U,\mu_0)
\]
is independent of $\tau\in[0,1]$ and that,
for every $\tau\in[0,1]$, 
\[
\mathrm{deg}(- \dvg[a_\tau(x,\nabla u)]
- \mathcal{G}_\tau(u),U,(1-t)\mu_0+t\mu_1)
\]
is independent of $t\in[0,1]$.
Then assertion~\ref{homotopyb} also follows.
\end{proof}
\begin{thm}
\label{thm:apriori}
Assume that
\[
a:\Omega\times(\R^N\times [0,1])\rightarrow \R^N\,,\qquad
\mathcal{G}:\Phi^{1,p}_0(\Omega)\times[0,1]\rightarrow L^1(\Omega)
\]
satisfy~\ref{u1}, \ref{u2}, \ref{u3} and~\ref{u4}.
Suppose also that
\begin{multline*}
a_\tau(x,\xi) = \tau\,
a_1\left(x,\tau^{-\frac{1}{p-1}}\,\xi\right)\,,\qquad
\mathcal{G}_\tau(u) = \tau\,
\mathcal{G}_1\left(\tau^{-\frac{1}{p-1}}\,u\right)\,, \qquad
\text{whenever $0 < \tau \leq 1$}\,,
\end{multline*}
and that the problem
\[
\begin{cases}
- \dvg[a_0(x,\nabla u)] - 
\mathcal{G}_0(u) = 0\,,
&\quad\text{in $\Omega$}\,,\\
u=0\,,
&\quad\text{on $\partial\Omega$}\,,
\end{cases}
\]
has no entropy solution 
$u\in \Phi^{1,p}_0(\Omega)\setminus\{0\}$.
\par
Then, for every $R>0$, there exists a bounded and open 
subset $U$ of $\Phi^{1,p}_0(\Omega)$ containing all entropy
solutions of
\[
\begin{cases}
- \dvg[a_\tau(x,\nabla u)] - \mathcal{G}_\tau(u) = \mu\,,
&\quad\text{in $\Omega$}\,,\\
u=0\,,
&\quad\text{on $\partial\Omega$}\,,
\end{cases}
\]
with $0\leq \tau \leq 1$, $\mu\in \mathcal{M}_b^p(\Omega)$
and $|\mu|(\Omega)\leq R$.
Moreover, 
\[
\mathrm{deg}(- \dvg[a_\tau(x,\nabla u)]
- \mathcal{G}_\tau(u),U,\mu)
\]
is independent of $\tau\in[0,1]$ and of
$\mu\in \mathcal{M}_b^p(\Omega)$ with $|\mu|(\Omega)\leq R$.
\end{thm}
\begin{proof}
First of all, we have
\[
a_0(x,t\xi) = t^{p-1}\,a_0(x,\xi)\,,\quad
\mathcal{G}_0(tu) = t^{p-1}\,\mathcal{G}_0(u)\quad
\text{whenever $t\geq 0$}\,.
\]
Let $(\tau_n)$ be a sequence in $[0,1]$, 
$(\mu_n)$ a sequence in $\mathcal{M}_b^p(\Omega)$ with 
$|\mu_n|(\Omega)\leq R$ and assume that 
$u_n\in \Phi^{1,p}_0(\Omega)$ satisfies
\[
\begin{cases}
- \dvg[a_{\tau_n}(x,\nabla u_n)] 
- \mathcal{G}_{\tau_n}(u_n) = \mu_n
&\quad\text{in $\Omega$}\,,\\
u_n=0
&\quad\text{on $\partial\Omega$}\,,
\end{cases}
\] 
in the entropy sense.
We claim that $(u_n)$ is bounded in $\Phi^{1,p}_0(\Omega)$.
Actually, according to Remark~\ref{rem:equiv}, there exist 
a bounded sequence $(w_n^{(0)})$ in $L^1(\Omega)$ and 
a bounded sequence $(w_n^{(1)})$ in $L^{p'}(\Omega;\R^N)$
such that $\mu_n = w_n^{(0)} - \dvg w_n^{(1)}$.
Without loss of generality, we may assume that
\[
\int_\Omega |\nabla[\varphi_p(u_n)]|^p\,dx \geq 1
\qquad\text{for all $n\in\N$}\,.
\]
According to Proposition~\ref{prop:tPhi},
there exists $\varrho_n \geq 1$ such that
\[
\int_\Omega \left|\nabla\left[\varphi_p\left(
\frac{u_n}{\varrho_n}\right)\right]\right|^p\,dx = 1
\qquad\text{for all $n\in\N$}
\]
and it is enough to prove that $(\varrho_n)$ is bounded.
Assume, for a contradiction, that up to a subsequence
$\varrho_n \to +\infty$.
Then, if we set
\[
z_n = \frac{u_n}{\varrho_n}\,,\quad
\sigma_n = \frac{\tau_n}{\varrho_n^{p-1}}\,,
\]
it follows that $0\leq\sigma_n \leq 1$ and that
$z_n$ satisfies
\[
\begin{cases}
- \dvg[a_{\sigma_n}(x,\nabla z_n)] 
- \mathcal{G}_{\sigma_n}(z_n) 
= \varrho_n^{-(p-1)}\,\mu_n 
&\quad\text{in $\Omega$}\,,\\
z_n=0
&\quad\text{on $\partial\Omega$}\,,
\end{cases}
\]
in the entropy sense,
with $\sigma_n \to 0$ and $\varrho_n^{-(p-1)}\,w_n^{(0)} \to 0$,
$\varrho_n^{-(p-1)}\,w_n^{(1)} \to 0$ in $L^1(\Omega)$ and
$L^{p'}(\Omega;\R^n)$, respectively.
From Lemma~\ref{lem:proper} we infer that, up to a subsequence,
$(z_n)$ is convergent in $\Phi^{1,p}_0(\Omega)$ to
some~$z$ satisfying
\[
\int_\Omega |\nabla[\varphi_p(z)]|^p\,dx = 1
\]
and 
\[
\begin{cases}
- \dvg[a_0(x,\nabla z)] 
- \mathcal{G}_0(z) = 0 
&\quad\text{in $\Omega$}\,,\\
z=0
&\quad\text{on $\partial\Omega$}\,,
\end{cases}
\]
in the entropy sense, whence a contradiction.
Therefore $(u_n)$ is bounded in $\Phi^{1,p}_0(\Omega)$ and
there exists a bounded and open subset $U$ of 
$\Phi^{1,p}_0(\Omega)$ with the required property.
\par
From Theorem~\ref{thm:homotopy} we infer that
\[
\mathrm{deg}(- \dvg[a_\tau(x,\nabla u)]
- \mathcal{G}_\tau(u),U,t\mu) =
\mathrm{deg}(- \dvg[a_0(x,\nabla u)]
- \mathcal{G}_0(u),U,0) 
\]
for all $\mu\in \mathcal{M}_b^p(\Omega)$ with $|\mu|(\Omega)\leq R$
and all $(\tau,t)\in[0,1]\times[0,1]$.
\end{proof}
%


\section{The positively homogeneous operator}
\label{sect:homogeneous}
Throughout this section we assume that $\Omega$ is also connected.
We are interested in the entropy solutions $u$ of the problem
\begin{equation}
\label{eq:as}
\begin{cases}
- \Delta_p u - \underline{\lambda} \,(u^+)^{p-1} +
\overline{\lambda} \,(u^-)^{p-1} = \mu
&\qquad\text{in $\Omega$}\,,\\
u=0
&\qquad\text{on $\partial\Omega$}\,,
\end{cases}
\end{equation}
when $\underline{\lambda} < \lambda_1 < \overline{\lambda}$ 
and $\mu\in\mathcal{M}_b^p(\Omega)$.
We aim to apply the results of the previous section with
\begin{gather*}
a(x,\xi) = |\xi|^{p-2}\xi\,,\\
\mathcal{G}(u) = \underline{\lambda} \,(u^+)^{p-1} -
\overline{\lambda} \,(u^-)^{p-1} 
\quad\text{or}\quad
\mathcal{G}(u) = \underline{\lambda} \,(u^+)^{p-1} \,.
\end{gather*}
\begin{lem}
\label{lem:supersol}
The following facts hold:
\begin{enumerate}[label={\upshape\alph*)}, align=parleft, 
widest=(a), leftmargin=*]
\item[\mylabel{supersola}{\ha}]
if $\lambda < \lambda_1$, then there is no entropy solution 
$u\in\Phi^{1,p}_0(\Omega)\setminus\{0\}$ of
\[
\begin{cases}
- \Delta_p u - \lambda \,(u^+)^{p-1} = 0
&\qquad\text{in $\Omega$}\,,\\
u=0
&\qquad\text{on $\partial\Omega$}\,;
\end{cases}
\]
\item[\mylabel{supersolb}{\hb}]
if $\mu\geq 0$, $\lambda \in\R$ and 
$u\in\Phi^{1,p}_0(\Omega)\setminus\{0\}$ 
is an entropy solution of
\[
\begin{cases}
- \Delta_p u - \lambda \,(u^+)^{p-1} = \mu
&\qquad\text{in $\Omega$}\,,\\
u=0
&\qquad\text{on $\partial\Omega$}\,,
\end{cases}
\]
then we have $\lambda \leq\lambda_1$ and
\[
\essinf_K u > 0
\qquad\text{for all compact $K\subseteq \Omega$}\,.
\]
\end{enumerate}
\end{lem}
\begin{proof}
Assertion~\ref{supersola} easily follows as, by
Theorem~\ref{thm:consistency}, the entropy solutions
coincide, in this case, with the weak solutions 
in $W^{1,p}_0(\Omega)$.
\par
To prove assertion~\ref{supersolb}, apply first
Theorem~\ref{thm:bb} with the test function $v=u^-$.
Taking into account Proposition~\ref{prop:tPhi},
it follows that $u\geq 0$ a.e. in $\Omega$.
\par
Now apply again 
Theorem~\ref{thm:bb} with the test function 
$v=\varrho_\varepsilon(u)w$, where
\[
\varrho_\varepsilon(s) = 
\begin{cases}
1 
&\qquad\text{if $s\leq 1$}\,,\\
\noalign{\medskip}
\displaystyle{
1 - \frac{s-1}{\varepsilon}}
&\qquad\text{if $1 < s < 1+\varepsilon$}\,,\\
\noalign{\medskip}
0 
&\qquad\text{if $s\geq 1+\varepsilon$}\,,
\end{cases}
\]
and $w\in C^\infty_c(\Omega)$ with $w\geq 0$.
Arguing as in the proof of~\cite[Theorem~3.2]{boccardo_orsina2020}, 
we deduce that $T_1(u)\in W^{1,p}_0(\Omega)\cap L^\infty(\Omega)$ 
satisfies in a weak sense
\[
- \Delta_p T_1(u) - 
\lambda\,\chi_{\{u\leq 1\}}\,T_1(u)^{p-1} \geq 0
\qquad\text{in $\Omega$}\,.
\]
From~\cite[Theorem~1.2]{trudinger1967-cpam} we infer that
\[
\essinf_K T_1(u) > 0
\qquad\text{for all compact $K\subseteq \Omega$}\,,
\]
whence
\[
\essinf_K u > 0
\qquad\text{for all compact $K\subseteq \Omega$}\,.
\]
Now  assume, for a contradiction, that $\lambda > \lambda_1$.
Then, there exists $k>0$ such that
\begin{equation}
\label{eq:k}
\lambda > 
\min\left\{\int_\Omega |\nabla w|^p\,d\leb^N:\,\,
w\in W^{1,p}_0(\Omega)\,,\,\,
\int_\Omega \chi_{\{u \leq k\}}\, |w|^p\,d\leb^N=1\right\}\,.
\end{equation}
Arguing as before, we infer that 
$\overline{u}:=T_k(u)\in W^{1,p}_0(\Omega)\cap L^\infty(\Omega)$ 
satisfies
\begin{gather*}
- \Delta_p \overline{u} - 
\lambda \,\chi_{\{u \leq k\}}\,
\overline{u}^{p-1} \geq 0
\qquad\text{in $\Omega$}\,,\\
\essinf_K \overline{u} > 0
\qquad\text{for all compact $K\subseteq \Omega$}\,.
\end{gather*}
Since $\overline{u}\in L^\infty(\Omega)$, there exists
a minimum $z$ of the functional
\[
w \mapsto \int_\Omega |\nabla w|^p\,d\leb^N -
\lambda\,
\int_\Omega  \chi_{\{u \leq k\}}\, w^p\,d\leb^N 
\]
on the convex set 
\[
\left\{w\in W^{1,p}_0(\Omega):\,\,
0 \leq w \leq\overline{u}\right\}\,.
\]
In principle, $z$ satisfies the variational inequality
\begin{multline*}
\int_\Omega \left[|\nabla z|^{p-2}\nabla z\cdot\nabla(w-z)
- \lambda  \chi_{\{u \leq k\}}\, 
z^{p-1}(w-z)\right]\,d\leb^N \geq 0 \\
\qquad\text{for all $w\in W^{1,p}_0(\Omega)$ with 
$0 \leq w \leq\overline{u}$}\,.
\end{multline*}
However, since $(0,\overline{u})$ 
is a pair of sub/super-solutions of
\begin{equation}
\label{eq:ovl}
- \Delta_p w
- \lambda\, \chi_{\{u \leq k\}}\, w^{p-1} = 0\,,
\end{equation}
from e.g.~\cite[Theorem~3.2]{degiovanni_marzocchi2019}
we infer that $z$ is a weak solution of~\eqref{eq:ovl}.
\par
By \eqref{eq:k} there exists
$w\in W^{1,p}_0(\Omega)\cap L^\infty(\Omega)$,
with compact support in $\Omega$ and 
$w \geq 0$ a.e in $\Omega$, such that
\[
\int_\Omega |\nabla w|^p\,d\leb^N - 
\lambda\,
\int_\Omega  \chi_{\{u \leq k\}}\, w^p\,d\leb^N < 0\,.
\]
Since $0\leq tw \leq\overline{u}$ if $t>0$ is small enough, we 
infer that
\[
\int_\Omega |\nabla z|^p\,d\leb^N - 
\lambda\,
\int_\Omega  \chi_{\{u \leq k\}}\, z^p\,d\leb^N < 0\,,
\]
so that $z$ is a nontrivial solution of~\eqref{eq:ovl}
with constant sign.
From~\cite{brasco_franzina2012, kawohl_lindqvist2006} we infer that
\[
\lambda =
\min\left\{\int_\Omega |\nabla w|^p\,d\leb^N:\,\,
w\in W^{1,p}_0(\Omega)\,,\,\,
\int_\Omega \chi_{\{u \leq k\}}\, |w|^p\,d\leb^N=1\right\}
\]
and a contradiction follows.
\end{proof}
\begin{thm}
\label{thm:nosol}
If $\mu \leq 0$, then there is no entropy solution $u$
of~\eqref{eq:as} with $u\in\Phi^{1,p}_0(\Omega)\setminus\{0\}$.
\end{thm}
\begin{proof}
Assume, for a contradiction, that
$u\in\Phi^{1,p}_0(\Omega)\setminus\{0\}$ is an entropy solution
of~\eqref{eq:as}.
A standard bootstrap argument (see 
also~\cite[Theorems~1 and~3]{boccardo_gallouet1992-cpde}) shows 
that $u^+ \in L^p(\Omega)$, whence $u^+ \in W^{1,p}_0(\Omega)$.
Since $\underline{\lambda}<\lambda_1$,  we infer that $u\leq 0$ 
a.e. in $\Omega$, so that 
$u^-\in\Phi^{1,p}_0(\Omega)\setminus\{0\}$ satisfies
\[
\begin{cases}
- \Delta_p u^- - \overline{\lambda} \,(u^-)^{p-1} = - \mu
&\qquad\text{in $\Omega$}\,,\\
u^-=0
&\qquad\text{on $\partial\Omega$}\,,
\end{cases}
\]
in the entropy sense.
From Lemma~\ref{lem:supersol} we infer that 
$\overline{\lambda} \leq \lambda_1$ and 
a contradiction follows.
\end{proof}
\begin{lem}
\label{lem:degree01h}
For every $\mu\in\mathcal{M}_b^p(\Omega)$,
there exists a bounded and open subset $U_0$ of 
$\Phi^{1,p}_0(\Omega)$ such that:
\begin{enumerate}[label={\upshape\alph*)}, align=parleft, 
widest=(a), leftmargin=*]
\item[\mylabel{degree01a}{\ha}]
the problems~\eqref{eq:as} and
\[
\begin{cases}
- \Delta_p u - \underline{\lambda} (u^+)^{p-1} = \mu 
&\qquad\text{in $\Omega$}\,,\\
u=0
&\qquad\text{on $\partial\Omega$}\,,
\end{cases}
\]
have no entropy solution 
$u\in \Phi^{1,p}_0(\Omega) \setminus U_0$;
\item[\mylabel{degree01b}{\hb}]
we have
\begin{alignat*}{3}
&\mathrm{deg}(- \Delta_p u - 
\underline{\lambda} \,(u^+)^{p-1} +
\overline{\lambda} \,(u^-)^{p-1},U_0,\mu) &&= 0\,,\\
&\mathrm{deg}(- \Delta_p u - \underline{\lambda} (u^+)^{p-1},
U_0,\mu) 
&&= 1\,.
\end{alignat*}
\end{enumerate}
\end{lem}
\begin{proof}
According to Theorem~\ref{thm:nosol} and Lemma~\ref{lem:supersol}, 
the problems
\[
\begin{cases}
- \Delta_p u - \underline{\lambda} (u^+)^{p-1} +
\overline{\lambda} (u^-)^{p-1} = 0
&\qquad\text{in $\Omega$}\,,\\
u=0
&\qquad\text{on $\partial\Omega$}\,,
\end{cases}
\]
\[
\begin{cases}
- \Delta_p u - \underline{\lambda} (u^+)^{p-1} = 0
&\qquad\text{in $\Omega$}\,,\\
u=0
&\qquad\text{on $\partial\Omega$}\,,
\end{cases}
\]
have no entropy solution $u\in \Phi^{1,p}_0(\Omega)\setminus\{0\}$.
From Theorem~\ref{thm:apriori} we infer that there exists a bounded
and open subset $U_0$ of $\Phi^{1,p}_0(\Omega)$ such that there is
no entropy solution of the problems
\[
\begin{cases}
- \Delta_p u - \underline{\lambda} (u^+)^{p-1} +
\overline{\lambda} (u^-)^{p-1} = (1-t)\mu - t
&\qquad\text{in $\Omega$}\,,\\
u=0
&\qquad\text{on $\partial\Omega$}\,,
\end{cases}
\]
\[
\label{eq:1}
\begin{cases}
- \Delta_p u - \underline{\lambda} (u^+)^{p-1} = t\mu 
&\qquad\text{in $\Omega$}\,,\\
u=0
&\qquad\text{on $\partial\Omega$}\,,
\end{cases}
\]
with $0\leq t\leq 1$ and $u\in \Phi^{1,p}_0(\Omega)\setminus U_0$,
and we have
\begin{multline*}
\mathrm{deg}(- \Delta_p u
- \underline{\lambda} (u^+)^{p-1} 
+ \overline{\lambda} (u^-)^{p-1},U_0,\mu) \\
= \mathrm{deg}(- \Delta_p u
- \underline{\lambda} (u^+)^{p-1} 
+ \overline{\lambda} (u^-)^{p-1},U_0,-1) \,,
\end{multline*}
\[
\mathrm{deg}(- \Delta_p u - \underline{\lambda} (u^+)^{p-1},
U_0,\mu) =
\mathrm{deg}(- \Delta_p u - \underline{\lambda} (u^+)^{p-1},
U_0,0)\,.
\]
Moreover, we have $0\in U_0$ and, again by 
Lemma~\ref{lem:supersol}, there is no entropy solution of
\[
\begin{cases}
- \Delta_p u - t \underline{\lambda} (u^+)^{p-1} = 0
&\qquad\text{in $\Omega$}\,,\\
u=0
&\qquad\text{on $\partial\Omega$}\,.
\end{cases}
\]
with $0\leq t\leq 1$ and $u\in \Phi^{1,p}_0(\Omega)\setminus U_0$,
whence
\[
\mathrm{deg}(- \Delta_p u - \underline{\lambda} (u^+)^{p-1},U_0,0)
= \mathrm{deg}(- \Delta_p u,U_0,0)
\]
by Theorem~\ref{thm:homotopy}.
\par
On the other hand, there is no entropy solution $u$ of
\[
\begin{cases}
- \Delta_p u - \underline{\lambda} (u^+)^{p-1} +
\overline{\lambda} (u^-)^{p-1} = - 1
&\qquad\text{in $\Omega$}\,,\\
u=0
&\qquad\text{on $\partial\Omega$}\,,
\end{cases}
\]
by Theorem~\ref{thm:nosol}.
From Theorems~\ref{thm:existence} and~\ref{thm:normalization}
we infer that
\begin{gather*}
\mathrm{deg}(- \Delta_p u
- \underline{\lambda} (u^+)^{p-1} 
+ \overline{\lambda} (u^-)^{p-1},U_0,-1) =0\,,\\
\mathrm{deg}(- \Delta_p u,U_0,0) = 1\,,
\end{gather*}
whence
\begin{gather*}
\mathrm{deg}(- \Delta_p u
- \underline{\lambda} (u^+)^{p-1} 
+ \overline{\lambda} (u^-)^{p-1},U_0,\mu) = 0\,,\\
\mathrm{deg}(- \Delta_p u
- \underline{\lambda} (u^+)^{p-1},U_0,\mu) = 1\,.
\end{gather*}
\end{proof}
We define
\[
P = \left\{u\in \Phi^{1,p}_0(\Omega):\,\,
\text{$u\geq 0$ a.e. in $\Omega$}\right\}\,.
\]
\begin{thm}
\label{thm:degreepos}
Let $\mu\in\mathcal{M}_b^p(\Omega)\setminus\{0\}$ with $\mu\geq 0$.
Then there exist a bounded and open subset $U_0$ and an open subset 
$U_1$ of $\Phi^{1,p}_0(\Omega)$ such that:
\begin{enumerate}[label={\upshape\alph*)}, align=parleft, 
widest=(a), leftmargin=*]
\item[\mylabel{degreeposa}{\ha}]
we have $P\subseteq U_1$ and there is no entropy
solution $u$ of~\eqref{eq:as} with 
\[
u\in \left(\Phi^{1,p}_0(\Omega) \setminus U_0\right) \cup
\left(\cl{U_1}\setminus P\right) \,;
\]
\item[\mylabel{degreeposb}{\hb}]
we have
\begin{alignat*}{3}
&\mathrm{deg}(- \Delta_p u
- \underline{\lambda} (u^+)^{p-1} 
+ \overline{\lambda} (u^-)^{p-1},U_0\cap U_1,\mu) 
&&= 1\,,\\
&\mathrm{deg}(- \Delta_p u
- \underline{\lambda} (u^+)^{p-1} 
+ \overline{\lambda} (u^-)^{p-1},U_0\setminus\cl{U_1},\mu) 
&&= - 1\,.
\end{alignat*}
\end{enumerate}
\end{thm}
\begin{proof}
Let $U_0$ be as in Lemma~\ref{lem:degree01h}, so that there is no 
entropy solution $u$ of~\eqref{eq:as} with 
$u\in\Phi^{1,p}_0(\Omega)\setminus U_0$.
We claim that there exists $n\geq 1$ such that there is 
no entropy solution $u \in \cl{U_0}$ of
\begin{equation}
\label{eq:tol}
\begin{cases}
- \Delta_p u - \underline{\lambda} (u^+)^{p-1} +
t\,\overline{\lambda} (u^-)^{p-1} = \mu
&\qquad\text{in $\Omega$}\,,\\
u=0
&\qquad\text{on $\partial\Omega$}\,,
\end{cases}
\end{equation}
with $0\leq t\leq 1$ and
\[
0 < \int_\Omega (u^-)^{p-1}\,d\leb^N \leq \frac{1}{n}\,.
\]
Actually assume, for a contradiction, that 
$0\leq t_n \leq 1$ and that $u_n\in \cl{U_0}$
satisfies
\begin{gather*}
\begin{cases}
- \Delta_p u_n - \underline{\lambda} (u_n^+)^{p-1} +
t_n\,\overline{\lambda} (u_n^-)^{p-1} = \mu
&\qquad\text{in $\Omega$}\,,\\
u_n=0
&\qquad\text{on $\partial\Omega$}\,,
\end{cases}
\\
0 < \int_\Omega (u_n^-)^{p-1}\,d\leb^N \leq \frac{1}{n}\,.
\end{gather*}
Since $U_0$ is bounded in $\Phi^{1,p}_0(\Omega)$, 
by Proposition~\ref{prop:Phi} and Theorem~\ref{thm:compres}
$(u_n)$ is convergent, up to a subsequence, to some $u$ in
$\Phi^{1,p}_0(\Omega)$ and a.e. in $\Omega$.
Then $u$ satisfies $u\geq 0$ a.e. in $\Omega$ and
\[
\begin{cases}
- \Delta_p u - \underline{\lambda} u^{p-1} = \mu
&\qquad\text{in $\Omega$}\,,\\
u=0
&\qquad\text{on $\partial\Omega$}\,,
\end{cases}
\]
in the entropy sense.
From Lemma~\ref{lem:supersol} we infer that
\[
\essinf_K u > 0
\qquad\text{for all compact $K\subseteq \Omega$}\,.
\]
In particular, for a.e. $x\in\Omega$, we have 
$u_n(x) > 0$ eventually as $n\to\infty$.
Now let
\[
v_n = \frac{u_n}{
\displaystyle{\left(\int_\Omega (u_n^-)^{p-1}\,d\leb^N
\right)^{1/(p-1)}}}\,.
\]
Then, for a.e. $x\in\Omega$, it holds
$v_n(x) > 0$ eventually as $n\to\infty$, whence
$v_n^-(x) = 0$ eventually as $n\to\infty$.
In particular, we have that
\begin{equation}
\label{eq:vn-}
\int_\Omega (v_n^-)^{p-1}\,d\leb^N = 1\,,\qquad
\lim_n v_n^- = 0\quad\text{a.e. in $\Omega$}\,,
\end{equation}
and $v_n$ satisfies
\[
\begin{cases}
\displaystyle{
- \Delta_p v_n - \underline{\lambda} (v_n^+)^{p-1} +
t_n\,\overline{\lambda} (v_n^-)^{p-1} = 
\frac{\mu}{\displaystyle{\int_\Omega (u_n^-)^{p-1}\,d\leb^N}}
}
&\qquad\text{in $\Omega$}\,,\\
v_n=0
&\qquad\text{on $\partial\Omega$}\,,
\end{cases}
\]
in the entropy sense.
If $\psi$ is given by Theorem~\ref{thm:regentr} and we use the test
function $\psi(v_n^-)$, from Theorem~\ref{thm:bb} we infer that
\[
\int_\Omega |\nabla [\varphi_p(v_n^-)]|^p\,d\leb^N =
\int_\Omega \psi'(v_n^-)|\nabla v_n^-|^p\,d\leb^N \leq
t_n\,\overline{\lambda}
\int_\Omega (v_n^-)^{p-1}\psi(v_n^-)\,d\leb^N\,.
\]
Therefore $(v_n^-)$ is bounded in $\Phi^{1,p}_0(\Omega)$ and
from Proposition~\ref{prop:Phi} we deduce that, up to
a subsequence, $((v_n^-)^{p-1})$ is strongly convergent
in $L^1(\Omega)$, which contradicts~\eqref{eq:vn-}.
Therefore, the claim is proved.
\par
If we set
\[
U_1 = \left\{u\in \Phi^{1,p}_0(\Omega):\,\,
\int_\Omega (u^-)^{p-1}\,d\leb^N < \frac{1}{n}\right\}\,,
\]
which is open in $\Phi^{1,p}_0(\Omega)$ according to 
Proposition~\ref{prop:Phi}, we have that~\eqref{eq:tol} 
has no entropy solution with $0\leq t\leq 1$ and
$u\in \left(\cl{U_1}\setminus P\right)\cap \cl{U_0}$.
In particular, assertion~\ref{degreeposa} follows.
On the other hand, problem~\eqref{eq:tol} has no entropy 
solution with $0\leq t\leq 1$ and 
$u\in \left(\Phi^{1,p}_0(\Omega)\setminus U_0\right)\cap P$.
From Theorem~\ref{thm:homotopy} we infer that
\begin{multline*}
\mathrm{deg}(- \Delta_p u
- \underline{\lambda} (u^+)^{p-1} 
+ \overline{\lambda} (u^-)^{p-1},U_0\cap U_1,\mu) 
= \mathrm{deg}(- \Delta_p u
- \underline{\lambda} (u^+)^{p-1},U_0\cap U_1,\mu) \,.
\end{multline*}
On the other hand, there is no entropy solution $u\not\in P$ of
\[
\begin{cases}
- \Delta_p u - \underline{\lambda} (u^+)^{p-1} = \mu
&\qquad\text{in $\Omega$}\,,\\
u=0
&\qquad\text{on $\partial\Omega$}\,,
\end{cases}
\]
by Lemma~\ref{lem:supersol}.
From Theorem~\ref{thm:addexc} we infer that
\[
\mathrm{deg}(- \Delta_p u
- \underline{\lambda} (u^+)^{p-1},U_0\cap U_1,\mu) =
\mathrm{deg}(- \Delta_p u
- \underline{\lambda} (u^+)^{p-1},U_0,\mu) \,.
\]
From Theorem~\ref{thm:addexc} we also deduce that
\begin{multline*}
\mathrm{deg}(- \Delta_p u - \underline{\lambda} (u^+)^{p-1} 
+ \overline{\lambda} (u^-)^{p-1},U_0,\mu) \\
= \mathrm{deg}(- \Delta_p u - \underline{\lambda} (u^+)^{p-1} 
+ \overline{\lambda} (u^-)^{p-1},U_0\cap U_1,\mu) \\
+ \mathrm{deg}(- \Delta_p u - \underline{\lambda} (u^+)^{p-1} 
+ \overline{\lambda} (u^-)^{p-1},U_0\setminus\cl{U_1},\mu) \,.
\end{multline*}
Taking into account Lemma~\ref{lem:degree01h},
we conclude that
\begin{alignat*}{3}
&\mathrm{deg}(- \Delta_p u
- \underline{\lambda} (u^+)^{p-1} 
+ \overline{\lambda} (u^-)^{p-1},U_0\cap U_1,\mu) 
&&= 1\,,\\
&\mathrm{deg}(- \Delta_p u - \underline{\lambda} (u^+)^{p-1} 
+ \overline{\lambda} (u^-)^{p-1},U_0\setminus\cl{U_1},\mu) 
&&= - 1\,.
\end{alignat*}
\end{proof}
%


\section{A result of asymptotic type}
\label{sect:asympt}
Let 
\[
g:\Omega\times(\R\times\R^N)\rightarrow\R
\]
be a Carath\'eodory function satisfying~\ref{g1} and~\ref{g2}.
\begin{thm}
\label{thm:asympt}
Assume that $\Omega$ is also connected.
Then, for every 
$\mu_0\in\mathcal{M}_b^p(\Omega)\setminus\{0\}$ with $\mu_0\geq 0$
and for every $\mu_1\in \mathcal{M}_b^p(\Omega)$, there exist
$\hat{t}, \tilde{t}\in\R$ such that the problem~\eqref{eq:main}
has at least two entropy solutions for all $t>\hat{t}$ and no
entropy solution for all $t<\tilde{t}$.
\end{thm}
\begin{proof}
According to Remark~\ref{rem:equiv}, there exist
$w^{(0)}\in L^1(\Omega)$ and $w^{(1)}\in L^{p'}(\Omega;\R^N)$
such that $\mu_1 = w^{(0)} - \dvg w^{(1)}$.
If we set
\begin{gather*}
a_\tau(x,\xi) = |\xi|^{p-2}\xi - \tau w^{(1)}(x)
\qquad\text{whenever $0\leq \tau \leq 1$}\,,\\
\noalign{\medskip}
\mathcal{G}_\tau(u) = 
\begin{cases}
\tau \left[ g(x,\tau^{- \frac{1}{p-1}}u,
\tau^{- \frac{1}{p-1}}\nabla u) + w^{(0)}\right]
&\qquad\text{if $0<\tau \leq 1$}\,,\\
\noalign{\medskip}
\underline{\lambda} (u^+)^{p-1} 
- \overline{\lambda} (u^-)^{p-1}
&\qquad\text{if $\tau = 0$}\,,
\end{cases}
\end{gather*}
it easily follows that 
assumptions~\ref{u1}, \ref{u2}, \ref{u3} and~\ref{u4}
are satisfied.
\par
Now we claim that there exist a bounded and open 
subset $U_0$ of $\Phi^{1,p}_0(\Omega)$, an open subset 
$U_1$ of $\Phi^{1,p}_0(\Omega)$ and $\hat{\tau}>0$ such that:
\begin{enumerate}[label={\upshape\alph*)}, align=parleft, 
widest=(a), leftmargin=*]
\item[\mylabel{lemdegree01a}{\ha}]
there is no entropy solution of
\begin{equation}
\label{eq:tau}
\begin{cases}
- \dvg[a_\tau(x,\nabla u)] - \mathcal{G}_\tau(u) 
= \mu_0
&\quad\text{in $\Omega$}\,,\\
u=0
&\quad\text{on $\partial\Omega$}\,,
\end{cases}
\end{equation}
with $0\leq \tau < \hat{\tau}$ and
$u\in\bd{U_0}\cup(\cl{U_0}\cap\bd{U_1})$;
\item[\mylabel{lemdegree01b}{\hb}]
we have 
\begin{alignat*}{3}
&\mathrm{deg}(- \dvg[a_\tau(x,\nabla u)]
- \mathcal{G}_\tau(u),U_0\cap U_1,\mu_0) 
&&= 1\,,\\
&\mathrm{deg}(- \dvg[a_\tau(x,\nabla u)]
- \mathcal{G}_\tau(u),U_0\setminus\cl{U_1},\mu_0) 
&&= - 1\,,\\
\end{alignat*}
whenever $0\leq \tau < \hat{\tau}$.
\end{enumerate}
Actually, by Theorem~\ref{thm:degreepos}, there 
exist $U_0$ and $U_1$ satisfying the assertions in the 
case $\tau=0$.
Since $\bd{U_0}\cup(\cl{U_0}\cap\bd{U_1})$ is closed and
bounded in $\Phi^{1,p}_0(\Omega)$, from Theorem~\ref{thm:homotopy}
we infer that there exists $\hat{\tau}$ with the required 
properties and the claim follows.
\par
If $0 \leq \tau < \hat{\tau}$,
from Theorem~\ref{thm:existence} we infer that~\eqref{eq:tau}
admits an entropy solution in $U_0\cap U_1$ and an
entropy solution in $U_0\setminus \cl{U_1}$.
\par
On the other hand, if $u$ is an entropy solution
of~\eqref{eq:tau} with $\tau>0$ and we set 
$z=\tau^{- \frac{1}{p-1}}\,u$,
it is easily seen that $z$ satisfies
\[
\begin{cases}
\displaystyle{
- \Delta_p z + \dvg\,w^{(1)} -
g(x,z,\nabla z) - w^{(0)}
= \frac{1}{\tau}\,\mu_0}
&\quad\text{in $\Omega$}\,,\\
z=0
&\quad\text{on $\partial\Omega$}\,,
\end{cases}
\]
in the entropy sense, namely
\[
\begin{cases}
\displaystyle{
- \Delta_p z =
g(x,z,\nabla z) + \frac{1}{\tau}\,\mu_0 + \mu_1}
&\quad\text{in $\Omega$}\,,\\
z=0
&\quad\text{on $\partial\Omega$}\,.
\end{cases}
\]
If we set $\hat{t} = \hat{\tau}^{-1}$, we infer
that, for every $t > \hat{t}$, there exist two 
entropy solutions of~\eqref{eq:main}.
\par
Assume now, for a contradiction, that there exist
$t_n \to +\infty$ and $z_n$ satisfying
\[
\begin{cases}
- \Delta_p z_n =
g(x,z_n,\nabla z_n) - t_n\,\mu_0 + \mu_1
&\quad\text{in $\Omega$}\,,\\
z_n=0
&\quad\text{on $\partial\Omega$}\,,
\end{cases}
\]
in the entropy sense, namely 
\[
\begin{cases}
- \dvg[a_{\tau_n}(x,\nabla u_n)] 
- \mathcal{G}_{\tau_n}(u_n) = - \mu_0
&\quad\text{in $\Omega$}\,,\\
u_n=0
&\quad\text{on $\partial\Omega$}\,,
\end{cases}
\]
with 
\[
\tau_n = t_n^{-1}\,,\qquad
u_n = \tau_n^{\frac{1}{p-1}}\,z_n\,.
\]
Since the problem
\[
\begin{cases}
- \dvg[a_{0}(x,\nabla u)] 
- \mathcal{G}_{0}(u) = 0
&\quad\text{in $\Omega$}\,,\\
u=0
&\quad\text{on $\partial\Omega$}\,,
\end{cases}
\]
has no entropy solution $u\in \Phi^{1,p}_0(\Omega)\setminus\{0\}$
by Theorem~\ref{thm:nosol},
from Theorem~\ref{thm:apriori} we infer that
there exists $R>0$ such that
\[
\int_\Omega |\nabla[\varphi_p(u_n)]|^p\,d\leb^N \leq R
\qquad\text{for all $n\in\N$}\,.
\]
From~\ref{homotopya} of Theorem~\ref{thm:homotopy} we deduce
that there exists an entropy solution $u$ of
\[
\begin{cases}
- \dvg[a_{0}(x,\nabla u)] 
- \mathcal{G}_{0}(u) = - \mu_0
&\quad\text{in $\Omega$}\,,\\
u=0
&\quad\text{on $\partial\Omega$}\,.
\end{cases}
\]
By Theorem~\ref{thm:nosol} a contradiction follows.
Therefore, there exists $\tilde{t}\in\R$ such 
that~\eqref{eq:main} has no entropy solution,
whenever $t<\tilde{t}$.
\end{proof}
%


\section{Proof of the main result}
\label{sect:proof}
Let again
\[
g:\Omega\times(\R\times\R^N)\rightarrow\R
\]
be a Carath\'eodory function satisfying~\ref{g1} and~\ref{g2}.
\par\bigskip
\noindent
\emph{Proof of Theorem~\ref{thm:main}.}
\par\noindent
Let $\Omega$ be connected and let
$\mu_0\in\mathcal{M}_b^p(\Omega)\setminus\{0\}$ with $\mu_0\geq 0$
and $\mu_1\in \mathcal{M}_b^p(\Omega)$.
If we set
\begin{alignat*}{3}
&\underline{t} &&= 
\sup\biggl\{\tilde{t}\in\R:\,\,&&\text{problem~\eqref{eq:main} 
admits} 
\biggr. \\ \biggl.
&&&&&\qquad\text{no entropy solution whenever 
$t<\tilde{t}$}\biggr\}\,,\\
&\overline{t} &&= 
\inf\biggl\{\hat{t}\in\R:\,\,&&\text{problem~\eqref{eq:main} admits}
\biggr. \\ \biggl.
&&&&&\qquad\text{at least two entropy solutions whenever 
$t>\hat{t}$}\biggr\}\,,
\end{alignat*}
it follows from Theorem~\ref{thm:asympt} that
\[
-\infty < \underline{t} \leq \overline{t} < +\infty
\]
and that
\begin{itemize}
\item
problem~\eqref{eq:main} admits no entropy solution whenever 
$t < \underline{t}$;
\item
problem~\eqref{eq:main} admits at least two entropy solutions 
whenever $t > \overline{t}$.
\end{itemize}
To conclude the proof, it is enough show that
\begin{itemize}
\item
problem~\eqref{eq:main} admits at least one entropy solution 
whenever $t\geq\underline{t}$.
\end{itemize}
Define $a_\tau(x,\xi)$ and $\mathcal{G}_\tau(u)$ as in the 
previous section, so 
that~\ref{u1}, \ref{u2}, \ref{u3} and~\ref{u4} are satisfied.
\par
There exist a sequence $t_n\to\underline{t}$ and a sequence 
$(u_n)$ in $\Phi^{1,p}_0(\Omega)$ satisfying
\[
\begin{cases}
- \Delta_p u_n =
g(x,u_n,\nabla u_n) + t_n\,\mu_0 + \mu_1
&\quad\text{in $\Omega$}\,,\\
u_n=0
&\quad\text{on $\partial\Omega$}\,,
\end{cases}
\]
in the entropy sense, namely
\[
\begin{cases}
- \dvg[a_1(x,\nabla u_n)] - \mathcal{G}_1(u_n) 
= t_n\,\mu_0
&\quad\text{in $\Omega$}\,,\\
u_n=0
&\quad\text{on $\partial\Omega$}\,.
\end{cases}
\]
Since the problem
\[
\begin{cases}
- \dvg[a_{0}(x,\nabla u)] 
- \mathcal{G}_{0}(u) = 0
&\quad\text{in $\Omega$}\,,\\
u=0
&\quad\text{on $\partial\Omega$}\,,
\end{cases}
\]
has no entropy solution $u\in \Phi^{1,p}_0(\Omega)\setminus\{0\}$
by Theorem~\ref{thm:nosol},
from Theorem~\ref{thm:apriori} we infer that
there exists $R>0$ such that
\[
\int_\Omega |\nabla[\varphi_p(u_n)]|^p\,d\leb^N \leq R
\qquad\text{for all $n\in\N$}\,.
\]
From~\ref{homotopya} of Theorem~\ref{thm:homotopy} we deduce
that there exists 
$\underline{u}\in  \Phi^{1,p}_0(\Omega)$ satisfying
\[
\begin{cases}
- \dvg[a_1(x,\nabla \underline{u})] 
- \mathcal{G}_1(\underline{u}) 
= \underline{t}\,\mu_0
&\quad\text{in $\Omega$}\,,\\
\underline{u}=0
&\quad\text{on $\partial\Omega$}\,,
\end{cases}
\]
namely
\[
\begin{cases}
- \Delta_p \underline{u} =
g(x,\underline{u},\nabla \underline{u}) 
+ \underline{t}\,\mu_0 + \mu_1
&\quad\text{in $\Omega$}\,,\\
\underline{u}=0
&\quad\text{on $\partial\Omega$}\,,
\end{cases}
\]
in the entropy sense.
Therefore, the assertion is true for $t = \underline{t}$.
\par
Let now $t > \underline{t}$ and define
\[
\widetilde{\mathcal{G}}_\tau(u) = 
\begin{cases}
\tau \left[ g\left(x,
\max\left\{\tau^{- \frac{1}{p-1}}u,\underline{u}\right\},
\nabla\max\left\{\tau^{- \frac{1}{p-1}}u,\underline{u}\right\}
\right) + w^{(0)}\right]
&\text{if $0<\tau \leq 1$}\,,\\
\noalign{\medskip}
\underline{\lambda} (u^+)^{p-1} 
&\text{if $\tau = 0$}\,.
\end{cases}
\]
We also have
\begin{gather*}
\widetilde{\mathcal{G}}_1(u) 
= \mathcal{G}_1(\max\{u,\underline{u}\}) \,,\\
\widetilde{\mathcal{G}}_\tau(u) = 
\tau\, \widetilde{\mathcal{G}}_1\left(\tau^{- \frac{1}{p-1}}u\right)
\qquad\text{if $0<\tau\leq 1$}\,.
\end{gather*}
Then it is easily seen that $\widetilde{\mathcal{G}}_\tau(u)$
also satisfies~\ref{u3} and~\ref{u4}.
\par
By Lemma~\ref{lem:supersol}, the problem
\[
\begin{cases}
- \dvg[a_{0}(x,\nabla u)] 
- \widetilde{\mathcal{G}}_{0}(u) = 0
&\quad\text{in $\Omega$}\,,\\
u=0
&\quad\text{on $\partial\Omega$}\,,
\end{cases}
\]
has no entropy solution 
$u\in \Phi^{1,p}_0(\Omega) \setminus \{0\}$.
From Theorem~\ref{thm:apriori} we infer that there exists
a bounded and open subset $U$ of $\Phi^{1,p}_0(\Omega)$ 
such that the problem
\[
\begin{cases}
- \dvg[a_{\tau}(x,\nabla u)] 
- \widetilde{\mathcal{G}}_{\tau}(u) = \tau t\mu_0
&\quad\text{in $\Omega$}\,,\\
u=0
&\quad\text{on $\partial\Omega$}\,,
\end{cases}
\]
has no entropy solution with $0\leq\tau\leq 1$ and 
$u\in\Phi^{1,p}_0(\Omega)\setminus U$.
Moreover, we have
\[  
\mathrm{deg}(- \dvg[a_1(x,\nabla u)]
- \widetilde{\mathcal{G}}_1(u),U,t\mu_0) 
=
\mathrm{deg}(- \dvg[a_0(x,\nabla u)]
- \widetilde{\mathcal{G}}_0(u),U,0) \,.
\]
On the other hand, by Lemma~\ref{lem:degree01h}
and Theorem~\ref{thm:addexc}, it holds
\[
\mathrm{deg}(- \dvg[a_0(x,\nabla u)]
- \widetilde{\mathcal{G}}_0(u),U,0) =
\mathrm{deg}(- \Delta_p u
- \underline{\lambda}(u^+)^{p-1},U,0) = 1\,,
\]
whence
\[  
\mathrm{deg}(- \dvg[a_1(x,\nabla u)]
- \widetilde{\mathcal{G}}_1(u),U,t\mu_0) = 1 \,.
\]
By Theorem~\ref{thm:existence} there exists an entropy
solution $u$ of
\[
\begin{cases}
- \dvg[a_{1}(x,\nabla u)] 
- \widetilde{\mathcal{G}}_{1}(u) = t\mu_0
&\quad\text{in $\Omega$}\,,\\
u=0
&\quad\text{on $\partial\Omega$}\,,
\end{cases}
\]
namely of
\[
\begin{cases}
- \Delta_p u =
g(x,\max\{u,\underline{u}\},
\nabla\max\{u,\underline{u}\}) + t\mu_0 + \mu_1
&\quad\text{in $\Omega$}\,,\\
u=0
&\quad\text{on $\partial\Omega$}\,.
\end{cases}
\]
By Theorem~\ref{thm:g=0}, for every $k>0$ we have
\begin{multline*}
\int_{\{0 \leq \underline{u} - u \leq k\}}
[a(x,\nabla \underline{u}) - a(x,\nabla u)]\cdot
(\nabla \underline{u} - \nabla u)\,d\leb^N \\
\leq
\int_\Omega [g(x,\underline{u},\nabla\underline{u}) -
g(x,\max\{u,\underline{u}\},\nabla\max\{u,\underline{u}\})]
[T_k(\underline{u} - u)]^+\,d\leb^N \\
- (t - \underline{t})
\int_\Omega [T_k(\underline{u} - u)]^+\,d\mu_0 \leq 0\,,
\end{multline*}
whence $\nabla \underline{u} = \nabla u$ a.e. in
$\{u \leq \underline{u}\}$, namely
\[
\nabla\max\{u,\underline{u}\} = \nabla u
\qquad\text{a.e. in $\Omega$}\,.
\]
From Proposition~\ref{prop:tPhi} we infer that
$\max\{u,\underline{u}\} = u$ a.e. in $\Omega$,
so that $u$ is an entropy solution of~\eqref{eq:main}.
\qed


%

\begin{thebibliography}{99}
%
\bibitem{amann_hess1979}
\au{H.~Amann and P.~Hess},
\titleart{A multiplicity result for a class of elliptic boundary 
value problems},
\jour{Proc. Roy. Soc. Edinburgh Sect. A} \volart{84} (1979),
\no{1-2}, 145--151.
%
\bibitem{amann_quittner1998}
\au{H.~Amann and P.~Quittner},
\titleart{Elliptic boundary value problems involving measures:
existence, regularity, and multiplicity},
\jour{Adv. Differential Equations} \volart{3} (1998),
\no{6}, 753--813.
%
\bibitem{ambrosetti1984}
\au{A.~Ambrosetti},
\titleart{Elliptic equations with jumping nonlinearities},
\jour{J. Math. Phys. Sci.} \volart{18} (1984),
\no{1}, 1--12.
%
\bibitem{ambrosetti_prodi1972}
\au{A.~Ambrosetti and G.~Prodi},
\titleart{On the inversion of some differentiable mappings 
with singularities between {B}anach spaces},
\jour{Ann. Mat. Pura Appl. (4)} \volart{93} (1972), 231--246.
%
\bibitem{ambrosetti_prodi1993}
\au{A.~Ambrosetti and G.~Prodi},
``A primer of nonlinear analysis'',
\textit{Cambridge Studies in Advanced Mathematics},
\textbf{34}, Cambridge University Press, Cambridge, 1993.
%
\bibitem{arcoya_ruiz2006}
\au{D.~Arcoya and D.~Ruiz},
\titleart{The {A}mbrosetti-{P}rodi problem for the 
{$p$}-{L}aplacian operator},
\jour{Comm. Partial Differential Equations} \volart{31} (2006),
\no{4-6}, 849--865.
%
\bibitem{benilan_boccardo_gallouet_gariepy_pierre_vazquez1995}
\au{P.~B{\'e}nilan, L.~Boccardo, T.~Gallou{\"e}t, R.~Gariepy, 
M.~Pierre, and J.~L.~V{\'a}zquez},
\titleart{An {$L^1$}-theory of existence and uniqueness of 
solutions of nonlinear elliptic equations},
\jour{Ann. Scuola Norm. Sup. Pisa Cl. Sci. (4)} \volart{22} (1995),
\no{2}, 241--273.
%
\bibitem{boccardo_gallouet1992-cpde}
\au{L.~Boccardo and T.~Gallou{\"e}t}, 
\titleart{Nonlinear elliptic equations with right hand side 
measures}, 
\jour{Comm. Partial Differential Equations} \volart{17} (1992),
\no{3-4}, 641--655.
%
\bibitem{boccardo_gallouet1996}
\au{L.~Boccardo and T.~Gallouet},
\titleart{Summability of the solutions of nonlinear elliptic 
equations with right-hand side measures},
\jour{J. Convex Anal.} \volart{3} (1996),
\no{2}, 361--365.
%
\bibitem{boccardo_gallouet_orsina1996}
\au{L.~Boccardo, T.~Gallou{\"e}t, and L.~Orsina},
\titleart{Existence and uniqueness of entropy solutions 
for nonlinear elliptic equations with measure data},
\jour{Ann. Inst. H. Poincar\'e Anal. Non Lin\'eaire}
\volart{13} (1996), \no{5}, 539--551.
%
\bibitem{boccardo_orsina2020}
\au{L.~Boccardo and L.~Orsina},
\titleart{Strong maximum principle for some quasilinear 
{D}irichlet problems having natural growth terms},
\jour{Adv. Nonlinear Stud.} \volart{20} (2020),
\no{2}, 503--510.
%
\bibitem{brasco_franzina2012}
\au{L.~Brasco and G.~Franzina},
\titleart{A note on positive eigenfunctions and hidden convexity},
\jour{Arch. Math. (Basel)} \volart{99} (2012),
\no{4}, 367--374.
%
\bibitem{brezis_browder1978}
\au{H.~Brezis and F.~E. Browder}, 
\titleart{Sur une propri\'et\'e des espaces de {S}obolev},
\jour{C. R. Acad. Sci. Paris S\'er. A-B}
\volart{287} (1978), \no{3}, A113--A115.
%
\bibitem{brezis_marcus_ponce2007}
\au{H.~Brezis, M.~Marcus, and A.~C. Ponce}, 
{\em Nonlinear elliptic equations with measures revisited}, 
in Mathematical aspects of nonlinear dispersive equations, 
vol.~163 of Ann. of Math. Stud., Princeton Univ. Press, Princeton,
NJ, 2007, 55--109.
%
\bibitem{browder1983}
\au{F.E.~Browder},
\titleart{Fixed point theory and nonlinear problems},
\jour{Bull. Amer. Math. Soc. (N.S.)} \volart{9} (1983),
\no{1}, 1--39.
%
\bibitem{colturato_degiovanni2016}
\au{M.~Colturato and M.~Degiovanni},
\titleart{A {F}redholm alternative for quasilinear elliptic 
equations with right-hand side measure},
\jour{Adv. Nonlinear Anal.} \volart{5} (2016),
\no{2}, 177--203.
%
\bibitem {dalmaso_murat_orsina_prignet1999}
\au{G.~Dal~Maso, F.~Murat, L.~Orsina, and A.Prignet},
\titleart{Renormalized solutions of elliptic equations with 
general measure data},
\jour{Ann. Scuola Norm. Sup. Pisa Cl. Sci. (4)} 
\volart{28} (1999), \no{4}, 741--808.
%
\bibitem{degiovanni_marzocchi2019}
\au{M.~Degiovanni and M.~Marzocchi},
\titleart{Quasilinear elliptic equations with natural growth 
and quasilinear elliptic equations with singular drift},
\jour{Nonlinear Anal.} \volart{185} (2019), 206--215.
%
\bibitem{degiovanni_scaglia2011}
\au{M.~Degiovanni and M.~Scaglia},
\titleart{A variational approach to semilinear elliptic 
equations with measure data},
\jour{Discrete Contin. Dyn. Syst.} \volart{31} (2011), 
\no{4}, 1233--1248.
%
\bibitem{deuel_hess1976}
\au{J.~Deuel and P.~Hess},
\titleart{A criterion for the existence of solutions of non-linear
elliptic boundary value problems},
\jour{Proc. Roy. Soc. Edinburgh Sect. A} \volart{74} (1974/75), 
49--54.
%
\bibitem{ferrero_saccon2007}
\au{A.~Ferrero and C.~Saccon},
\titleart{Existence and multiplicity results for semilinear
elliptic equations with measure data and jumping nonlinearities},
\jour{Topol. Methods Nonlinear Anal.} \volart{30} (2007),
\no{1}, 37--65.
%
\bibitem{greco_iwaniec_sbordone1997}
\au{L.~Greco, T.~Iwaniec, and C.~Sbordone},
\titleart{Inverting the {$p$}-harmonic operator},
\jour{Manuscripta Math.} \volart{92} (1997),
\no{2}, 249--258.
%
\bibitem{groli_squassina2003}
\au{A.~Groli and M.~Squassina},
\titleart{On the existence of two solutions for a general 
class of jumping problems},
\jour{Topol. Methods Nonlinear Anal.} \volart{21} (2003),
\no{2}, 325--344.
%
\bibitem{kawohl_lindqvist2006}
\au{B.~Kawohl and P.~Lindqvist},
\titleart{Positive eigenfunctions for the {$p$}-{L}aplace 
operator revisited},
\jour{Analysis (Munich)} \volart{26} (2006),
\no{4}, 545--550.
%
\bibitem{kilpelainen_xu1996}
\au{T.~Kilpel{\"a}inen and X.~Xu},
\titleart{On the uniqueness problem for quasilinear 
elliptic equations involving measures},
\jour{Rev. Mat. Iberoamericana} \volart{12} (1996),
\no{2}, 461--475.
%
\bibitem{koizumi_schmitt2005}
\au{E.~Koizumi and K.~Schmitt},
\titleart{Ambrosetti-{P}rodi-type problems for quasilinear 
elliptic problems},
\jour{Differential Integral Equations} \volart{18} (2005),
\no{3}, 241--262.
%
\bibitem{lindqvist1990}
\au{P.~Lindqvist},
\titleart{On the equation 
{${\rm div}\,(\vert \nabla u\vert \sp{p-2}\nabla u)
+\lambda\vert u\vert \sp {p-2}u=0$}},
\jour{Proc. Amer. Math. Soc.} \volart{109} (1990),
\no{1}, 157--164.
%
\bibitem{marcus_veron2014}
\au{M.~Marcus and L~V{\'e}ron},
``Nonlinear second order elliptic equations involving measures'',
\textit{De Gruyter Series in Nonlinear Analysis and Applications},
\textbf{21}, De Gruyter, Berlin, 2014.
%
\bibitem{marino_micheletti_pistoia1994}
\au{A.~Marino, A.~M.~Micheletti and A.~Pistoia},
\titleart{A nonsymmetric asymptotically linear elliptic problem},
\jour{Topol. Methods Nonlinear Anal.} \volart {4} (1994),
\no{2}, 289--339.
%
\bibitem{marino_saccon1997}
\au{A.~Marino and C.~Saccon},
\titleart{Some variational theorems of mixed type and elliptic 
problems with jumping nonlinearities},
\jour{Ann. Scuola Norm. Sup. Pisa Cl. Sci. (4)} \volart{25} (1997),
\no{3-4}, 631--665.
%
\bibitem{oregan_cho_chen2006}
\au{D.~O'Regan, Y.J.~Cho, and Y.-Q.~Chen},
``Topological degree theory and applications'',
\textit{Series in Mathematical Analysis and Applications},
\textbf{10}, Chapman \& Hall/CRC, Boca Raton, FL, 2006.
%
\bibitem{orsina1993}
\au{L.~Orsina}, 
\titleart{Solvability of linear and semilinear eigenvalue problems
with {$L\sp 1$} data}, 
\jour{Rend. Sem. Mat. Univ. Padova} \volart{90} (1993), 
207--238.
%
\bibitem{skrypnik1994}
\au{I.V.~Skrypnik},
``Methods for analysis of nonlinear elliptic 
boundary value problems'',
\textit{Translations of Mathematical Monographs}, \textbf{139},
American Mathematical Society, Providence, RI, 1994.
%
\bibitem{stampacchia1965}
\au{G.~Stampacchia},
\titleart{Le probl\`eme de {D}irichlet pour les \'equations
elliptiques du second ordre \`a coefficients discontinus},
\jour{Ann. Inst. Fourier (Grenoble)} \volart{15} (1965),
\no{1}, 189--258.
%
\bibitem{trudinger1967-cpam}
\au{N.~S.~Trudinger},
\titleart{On {H}arnack type inequalities and their application 
to quasilinear elliptic equations},
\jour{Comm. Pure Appl. Math.} \volart{20} (1967), 721--747.
%
\end{thebibliography}
\end{document}